\RequirePackage{fix-cm}
\documentclass[smallextended]{svjour3}       
\smartqed 

\usepackage{amsmath}
\usepackage{amssymb}
\usepackage[a4paper, total={6in, 8in}]{geometry}
\usepackage{algorithm}
\usepackage[dvipsnames]{xcolor}
\usepackage{caption}
\usepackage{subcaption}
\usepackage{algorithmicx}
\usepackage{url}
\usepackage{graphicx}
\usepackage{algpseudocode}
\newcommand{\F}{\mathcal{F}}

\newcommand{\DSB}{\tx{DSB}}

\newcommand{\Sc}[2]{\langle #1, #2 \rangle}

\newcommand{\m}{m}

\newcommand{\tx}{\textnormal}

\newcommand{\n}[1]{\|#1\|}
\newcommand{\argmin}{\tx{argmin}}
\newcommand{\argmax}{\tx{argmax}}
\newcommand{\dist}{\tx{\dist}}
\newcommand{\s}[2]{\langle #1, #2 \rangle}
\newcommand{\OM}{\mathcal{C}}

\newcommand{\conv}{\tx{conv}}
\newcommand{\sm}{\setminus}
\newcommand{\SB}{\tx{SB}}

\newcommand{\rr}{\mathbb{R}}

\newcommand{\bA}{\bar{A}}

\newcommand{\supp}{\text{supp}}

\newcommand{\cM}{\mathcal{M}}
\newcommand{\ck}{\sqrt{2\mu}}
\newcommand{\mE}{\mathbb{E}}

\newcommand{\cA}{\mathcal{A}}

\newcommand{\BB}{\bar{B}}

\newcommand{\SSC}{\textnormal{SSC}}
\newtheorem{assumption}{Assumption}
\def\Ab{{\mathsf A}}

\def\Ib{{\mathsf I}}

\def\Qb{{\mathsf Q}}

\def\b{\mathbf b}

\def\d{\mathbf d}
\def\e{\mathbf e}

\def\g{\mathbf g}

\def\oo{\mathbf o}

\def\q{{\mathbf q}}

\def\ss{\mathbf s}

\def\w{\mathbf w}
\def\x{\mathbf x}
\def\y{\mathbf y}
\def\z{\mathbf z}

\newcommand{\francesco}[1]{\textcolor{black}{#1}} 
\newcommand{\damiano}[1]{\textcolor{black}{#1}} 

\newcommand{\Prob}{\mathbb{P}}
\newcommand{\bx}{\bar{\x}}
\title{Projection free methods on product domains}
\author{ Immanuel Bomze
	\and
	Francesco~Rinaldi
	\and 
	Damiano~Zeffiro}

\institute{
	Immanuel Bomze \at
	Vienna Center of OR and Research Network Data Science, Universit\"{a}t Wien \\
	\email{immanuel.bomze@univie.ac.at}
	\and
	Francesco Rinaldi \at
	Department of Mathematics, University of Padua\\
	\email{rinaldi@math.unipd.it}     
	\and
	Damiano Zeffiro \at
	Department of Mathematics, University of Padua\\
	\email{damiano.zeffiro@studenti.unipd.it}  
}

\begin{document}
	\maketitle

	\begin{abstract}
		Projection-free block-coordinate methods avoid high computational cost per iteration, and at the same time exploit the particular problem structure of product domains. Frank-Wolfe-like approaches rank among the most popular ones of this type. However, as observed in the literature, there was a gap between the classical Frank-Wolfe theory and the block-coordinate case, \damiano{with no guarantees of linear convergence rates even for strongly convex objectives in the latter}. Moreover, most of previous research concentrated on convex objectives. This study now deals also with the non-convex case and reduces above-mentioned theory gap, in combining a new, fully developed convergence theory with novel active set identification results which ensure that inherent sparsity of solutions can be exploited in an efficient way. Preliminary numerical experiments seem to justify our approach and also show promising results for obtaining global solutions in the non-convex case.   
	\end{abstract}
	
	\maketitle
	
	\section{Introduction}\label{sec1}

 We consider the problem
	\begin{equation} \label{problem}
		\min_{\x \in \OM} f(\x) \, ,
	\end{equation}
	with objective $f$ having $L$-Lipschitz regular gradient, and feasible set $\OM\subseteq \rr^n$ closed and convex. Furthermore, we assume that $\OM$ is block separable, that is 
	\begin{equation} \label{d:Omega}
		\OM = \OM_{(1)} \times ... \times \OM_{(\m)}
	\end{equation}
	  with $\OM_{(i)} \subset \rr^{n_i}$ closed and convex for $ i \in [1\! : \! m]$, and of course $\sum_{i=1}^m n_i=n$.
 
	Notice that problem \eqref{problem} falls in the class of composite optimization problems
	\begin{equation}
		\min_{\x \in \OM} [f(\x) + g(\x)]
	\end{equation}
	with $f$ smooth and $g(\x) =\sum_{i=1}^m \chi_{\OM_{(i)}}(\x^{(i)})$ convex and block separable (see, e.g.,  \cite{richtarik2014iteration} for an overview of methods for this class of problems); here $\chi_D : \rr^d \to [0,+\infty ]$ denotes the indicator function of a convex set $D\subseteq \rr^d$, and for a block vector 
 $\x \in \rr^{n} = \rr^{n_1}\times ... \times \rr^{n_m}$ we denote by $\x^{(i)} \in \rr^{n_i}$ the component corresponding to the $i$-th block, so that $\x = (\x^{(1)}, ..., \x^{(m)})$.  
 
  Problems of this type arise in a wide number of real-world applications like, e.g., 
  traffic assignment~\cite{leblanc1975efficient}, structural
SVMs~\cite{lacoste2012block}, trace-norm based tensor completion
\cite{liu2012tensor}, reduced rank nonparametric regression
\cite{foygel2012nonparametric}, semi-relaxed optimal transport \cite{fukunaga2021fast}, structured submodular minimization \cite{jegelka2013reflection},   group fused lasso \cite{alaiz2013group}, and dictionary learning \cite{boumal2020introduction}. 

Block-coordinate gradient descent (BCGD) strategies (see, e.g., \cite{beck2017first}) represent a standard approach to solve problem \eqref{problem} in the convex case. When dealing with non-convex objectives, those methods can anyway still be used as an efficient tool to perform local searches  in probabilistic global optimization frameworks  (see, e.g., \cite{locatelli2013global} for further details).
The way BCGD approaches work is very easy to understand: those methods build up, at each iteration, a suitable model 
of the original function for a block of variables and then perform a projection on the feasible set related to that block. 

\francesco{Projection-based strategies (see, e.g., \cite{birgin2000nonmonotone,calamai1987projected,di2018two,di2023stationarity} for further details), albeit being widely used in practice also in a block-coordinate fashion (see, e.g., \cite{nesterov2012efficiency}), might however be costly even when the projection is performed over some structured  sets like, e.g., the flow polytope, the nuclear-norm ball, the Birkhoff polytope, or the permutahedron (see, e.g., \cite{combettes2021complexity}).} This is the reason why, in recent years, projection-free methods (see, e.g., \cite{bomze2021frank,jaggi2013revisiting,lan2020first}) have been massively used when dealing with those structured constraints. 

These methods simply rely on a suitable oracle  that minimizes, at each iteration, a linear approximation of the function over the original feasible set,  returning a point in
$$\argmin_{\x \in \OM} \s{\g}{\x}.$$
When $\OM$ is defined as in \eqref{d:Omega}, this decomposes in $m$ independent problems thanks to the block separable structure of the feasible set. In turn, the resulting problems on the blocks can then be solved in parallel, a possibility that has widely been explored in the literature, especially in the context of traffic assignment (see, e.g., \cite{leblanc1975efficient}).  In a big data context, performing a full update of the variables might still represent a computational bottleneck that needs to be properly handled in practice. This is the reason why block-coordinate variants of the classic Frank-Wolfe (FW) method have been recently proposed (see, e.g., \cite{lacoste2012block,osokin2016minding,wang2016parallel}). 
This method  is proposed in~\cite{lacoste2012block}
for structured support vector machine training, and randomly selects a block at each iteration to perform an FW update on the block. Several improvements on this algorithm, e.g., adaptive block sampling, use of pairwise and away-step directions, or oracle call caching, are described in \cite{osokin2016minding}, which obviously 
work in a sequential fashion. 

However, in case one wants to take advantage of modern multicore architectures or of distributed clusters, parallel and distributed versions of the block-coordinate FW algorithm are also available \cite{wang2016parallel}. It is important to highlight that all the papers mentioned above only consider convex programming problems and use random sampling variants as the main block selection strategy.

Furthermore, as noticed in \cite{osokin2016minding}, the standard convergence analysis for FW variants (e.g., pairwise and away step
FW) cannot be easily extended to the block-coordinate case. \damiano{In particular, there has been no extension in this setting of the well known linear convergence rate guarantees for FW variants applied to strongly convex objectives (see \cite{bomze2021frank} and references therein).}
This is mainly 
due to the difficulties in handling the bad/short steps (i.e., those steps  that do not give a good progress and are taken to guarantee feasibility of the iterate) within a block-coordinate framework. \damiano{In~\cite{osokin2016minding}, the authors hence extend the convergence analysis of FW variants to the block coordinate setting under the strong assumption that there are no bad steps, claiming that novel}
proof techniques are required to carry out the analysis in general and close the gap between FW and BCFW in this context.

Here we focus on the non-convex case and define a new general block-coordinate algorithmic framework that gives flexibility in the use of both block selection strategies and FW-like directions. Such a flexibility is mainly obtained thanks to  the way we perform approximate minimizations in the blocks. At each iteration, after selecting  one  block at least, we indeed use the Short Step Chain (SSC) procedure described in \cite{rinaldi2020avoiding}, which skips gradient computations in consecutive short steps until proper conditions are satisfied, to get the approximate minimization done in the selected blocks. 

Concerning the block selection strategies, we explore three different options. The first one we consider is a parallel or Jacobi-like strategy (see, e.g., \cite{bertsekas2015parallel}), where the SSC procedure is performed for all blocks. This obviously reduces the computational burden with respect to the use of the SSC in the whole variable space (see, e.g.,  \cite{rinaldi2020avoiding}) and eventually enables to use multicore architectures to perform those tasks in parallel. The second one is the random sampling (see, e.g.,   \cite{lacoste2012block}), where the SSC procedure is performed at each iteration on  a randomly selected subset of blocks. Finally we have a variant of the Gauss-Southwell rule (see, e.g., \cite{luo1992convergence}), where we perform SSC in all blocks and then select a block which violates optimality conditions at most. Such a greedy rule may make more progress in the objective function, since it uses first order information to choose the right block, but is, in principle, more expensive than  the other options we mentioned before (notice that the SSC is performed, at each iteration, for all blocks). 

Furthermore, we consider the following projection-free strategies: Away-step Frank-Wolfe (AFW), Pairwise Frank-Wolfe (PFW), and Frank-Wolfe
method with in face directions (FDFW), see, e.g., \cite{rinaldi2020avoiding} and references therein for further details.
The AFW and PFW strategies  depend on a set of ``elementary atoms'' $A$ such that $\OM = \conv(A)$. Given $A$, for a base point $\x \in \OM$ we can define
$$S_\x = \{S \subset A : \x \tx{ is a proper convex combination of all the elements in }S \} \, ,$$ the family of possible active sets for a given point $\x$.  For $\x \in \OM$ and $S \in S_\x$, $\d^{\tx{PFW}}$ is a PFW direction with respect to the active set $S$ and gradient $-\g$ if and only if
\begin{equation} \label{eq:PFW}
	\d^{\tx{PFW}} = \ss - \q \textnormal{ with } \ss \in \argmax_{\ss \in \OM} \s{\ss}{\g} \textnormal{ and } \q \in \argmin_{\q \in S} \s{\q}{\g} \, .
\end{equation}
Similarly, given $\x \in \OM$ and $S \in S_\x$, $\d^{\tx{AFW}}$ is an AFW direction with respect to the active set $S$ and gradient $-\g$ if and only if
\begin{equation} \label{AFWdir}
	\d^{\tx{AFW}} \in \argmax \{\s{\g}{\d} : \d \in \{\d^{\tx{FW}}, \d^{\tx{AS}}\} \} \, ,
\end{equation}
where $\d^{\tx{FW}}$ is a classic Frank-Wolfe direction
\begin{equation} \label{eq:FWstep}
	\d^{\tx{FW}}  = \ss - \x  \textnormal{ with } \ss \in {\argmax}_{\ss \in \OM} \s{\ss}{\g} \, ,
\end{equation}
and $\d^{\tx{AS}}$ is the away direction 
\begin{equation} \label{eq:ASstep}
	\d^{\tx{AS}}  = 	\x-\q  \textnormal{ with } \q \in \argmin_{\q \in S} \s{\q}{\g} \,  .
\end{equation}

The FDFW only requires the current point $\x$ and   gradient $-\g$ to select a descent direction (i.e., it does not need to keep track of the active set) and is defined as
\begin{equation*}
	\d^{F} = \x - \x_{F} \tx{ with } \x_{F} \in \argmin\{\s{\g}{\y} : \y \in \F(\x) \} 
\end{equation*}
for $\F(\x)$ the minimal face of $\OM$ containing $\x$. 
The selection criterion is then analogous to the one used by the AFW:
\begin{equation} \label{crit:FD}
	\d^{\tx{FD}} \in \tx{argmax} \{ \s{\g}{\d} : \d \in \{\d^{F}, \d^{\tx{FW}} \} \} \, .
\end{equation}

From a theoretical point of view, this new algorithmic framework enables us to give:
\begin{itemize}
\item  a local linear convergence rate for any choice of block selection strategy and FW-like direction. This result is obtained under a Kurdyka-\L ojasiewicz (KL)  property (see, e.g., \cite{attouch2010proximal}, \cite{bolte2007clarke} and \cite{bolte2010characterizations}) and a tailored angle condition (see, e.g., \cite{rinaldi2020avoiding}). Thanks to the way we handle short steps in our framework we are thus able to extend the analysis given for FW variants to the block-coordinate case and then 
to close the relevant gap in the theory  highlighted in \cite{osokin2016minding}. 
\item a local active set identification result (see, e.g., \cite{bomze2019first,bomze2020active,bomze2022fast,garber2020revisiting})  for a specific structure of the Cartesian product defining  the feasible set $\OM$, suitable choices of projection-free strategy (i.e., AFW direction is used), \damiano{and general smooth non convex objectives}. In particular, we prove that our framework identifies in finite time the support of a solution. Such a theoretical feature allows to reduce the
dimension of the problem at hand and, consequently, the overall computational cost of the optimization procedure. 
\end{itemize}

This is, to the best of our
knowledge, the first time that both a
(bad step free) linear convergence rate and an active set identification result are given for block-coordinate FW variants. In particular, we solve the open question from \cite{osokin2016minding} discussed above, \damiano{proving that the linear convergence rate of FW variants can indeed be extended to the block coordinate setting}. \damiano{Furthermore, our results guarantee,  for the first time in the literature of projection free optimization methods, identification  of the local active set in a {\sl single iteration} without a tailored active set strategy.}

We also report some preliminary numerical results on a specific class of structured problems with a block separable feasible set. Those results show that the proposed framework 
outperforms the classic block-coordinate FW and, thanks to its flexibility,
it can be effectively embedded into a probabilistic global optimization framework thus significantly boosting  its performances.

The paper is organized as follows.
Section \ref{s:pdalgorithm} describes the details of our new algorithmic framework. An in-depth analysis of its convergence properties is reported in Section \ref{s:crates}. An active set identification result is reported in Section \ref{s:actid}. Preliminary numerical results, focusing on the computational analysis of both the local identification and  the convergence properties of our framework, are reported in Section \ref{sec:numres}. Finally, some concluding remarks are included in Section  \ref{sec:conc}.

\subsection{Notation}

For a closed and convex set $C \subset \rr^h$  we denote by $\pi(C, \x)$ the projection of  $\x \in \rr^h$ onto $C$, and by $T_{C}(\x)$ the tangent space to $C$ at $\x\in C$. For $\g \in \rr^h$ we also use $\pi_\x(\g)$ as a shorthand for $\n{\pi(T_{C}(\x), \g)}$. We denote by $\hat{\y}$ the vector $\frac{\y}{\n{\y}}$ for $\y \neq \oo$, and $\hat{\y}=\oo$ otherwise. We finally denote by $\bar{B}_r(\x)$ and $B_r(\x)$ the closed and open balls of radius $r$ centered at $\x$.  

\section{A new block-coordinate projection-free  method } \label{s:pdalgorithm}

The block-coordinate framework we consider here applies the Short Step Chain (SSC) procedure from \cite{rinaldi2020avoiding}, described below as Algorithm~\ref{algo4}, to some of the blocks at every iteration. A detailed scheme is specified as Algorithm~\ref{algo6}; recall notation  $\x = (\x^{(1)}, ..., \x^{(m)})$ with $\x^{(i)}\in \OM_{(i)}$, all $i\in [1\! : \! m]$.

\begin{algorithm}
	\caption{Block coordinate method with  Short Step Chain (SSC) procedure}\label{algo6}
	\begin{algorithmic}[1]
		\State $\x_0 \in \OM$, $k = 0$.     
		\State If $\x_k$ is stationary, then STOP \label{algos5}
		\State Choose $ \mathcal{M}_k \subset [1 \! : \! m]$.
		\State For all $i \notin \mathcal{M}_k$ set $\x_{k + 1}^{(i)} = \x_k^{(i)}$
		\State For all $i \in \mathcal{M}_k$ set $\x_{k + 1}^{(i)} = \textnormal{SSC}(\x_k^{(i)},  -\nabla f(\x_k)^{(i)}))$
		\State $k = k+1$. Go to step \ref{algos5}.  
	\end{algorithmic}
\end{algorithm}
\par
In Algorithm~\ref{algo6}, we perform two main operations at each iteration. First, in Step 3,
we pick a suitable subset of blocks $\mathcal{M}_k$ according to a given block selection strategy. We then update  (Steps 4 and 5) the variables related to the selected blocks by means of the SSC procedure, while keeping all the variables in the other blocks unchanged.  

We now briefly recall the  SSC procedure from \cite{rinaldi2020avoiding}, designed to recycle the gradient in consecutive bad steps until suitable stopping conditions are met, in Algorithm \ref{algo4}.

\begin{algorithm}
	\caption{Short Step Chain procedure -- $\tx{SSC}(\bar \x,\g)$}\label{algo4}
	\begin{algorithmic}[1]
		\State \textbf{Initialization.} $\y_0 = \bx$, $j=0$ 
		\State Select $ \d_j \in \cA(\y_j, \g)$, $\alpha^{(j)}_{\max} \in \alpha_{\max}(\y_j, \d_j)$ \label{linealgo4}
		\If{$\d_j = 0$} 
		\Return  $\y_j$ 
		\EndIf
		\State compute an auxiliary step size $\beta_j$				 		 		
		\State let $\alpha_j = \min(\alpha^{(j)}_{\max}, \beta_j)$ 			
		\State $\y_{j+1} = \y_j + \alpha_j \d_j$ 
		\If{$\alpha_j =  \beta_j$}
		\Return $\y_{j+1}$
		\EndIf
		\State $j = j+1$, go to Step \ref{linealgo4} 
	\end{algorithmic}
\end{algorithm}
By $\mathcal{A}$ we indicate a projection-free strategy to generate first-order feasible descent directions for smooth functions on the block where the SSC is applied (e.g., FW, PFW, AFW directions). Since the gradient, $-\g$, is constant during the SSC procedure, it is easy to see that the procedure represents an application of $\cA$ to  minimize the linearized objective $f_\g(\z) = \s{ - \g}{\z - \bar{\x}} + f(\bar{\x})$, with suitable stepsizes and stopping condition. More specifically, after a stationarity check (see Steps 2--4), the stepsize $\alpha_j$ is the minimum of an auxiliary stepsize $\beta_j>0$ and the maximal stepsize $\alpha^{(j)}_{\max}$ (which we always assume to be strictly positive). The point $\y_{j + 1}$ generated at Step 7 is always feasible since $\alpha_j \leq \alpha^{(j)}_{\max}$. Notice that if the method $\mathcal{A}$ used in the SSC performs a FW step (see equation \eqref{eq:FWstep} for the definition of FW step), then the SSC terminates, with $\alpha_j = \beta_j$ or with $\y_{j + 1}$ a global minimizer of $f_\g$.

The auxiliary step size $\beta_j$ (see Step 5 of the SSC procedure) is thus defined as the maximal feasible stepsize (at $\y_j$) for the trust region
\begin{equation} \label{eq:omegaaux}
	\Omega_j =   B_{\n{\g}/2L}(\bar{\x} + \frac{\g}{2L}) \cap B_{\s{\g}{\hat{\d}_j}/L}(\bar{\x})  \, .
\end{equation} 
This guarantees the sufficient decrease condition
\begin{equation} \label{eq:dcondy}
	f(\y_j)  \leq f(\x_k) - \frac{L}{2}\n{\x_k - \y_j}^2 
\end{equation}	
and hence a monotone decrease of $f$ in the SSC.  For further details see \cite{rinaldi2020avoiding}.

\subsection{Block selection strategies}

As briefly mentioned in the introduction, we will consider three different block selection strategies in our analysis. The first one is a parallel or Jacobi-like strategy (see, e.g., \cite{bertsekas2015parallel}). In this case, we select all the blocks at each iteration. As we already observed, this is computationally cheaper than handling the whole variable space at once. Furthermore,  multicore architectures might eventually be considered to perform those tasks in parallel.
 A definition of the strategy is given below:

 \begin{definition}[Parallel selection] \label{d:parallel}
	Set $\cM_k = [1 \! : \! m]$.
\end{definition}

The second strategy is a variant of the GS rule (see, e.g., \cite{luo1992convergence}), where we first perform SSC in all blocks and then select a block that violates optimality conditions at most. The formal definition is reported below.

\begin{definition}[Gauss-Southwell (GS) selection] \label{d:GS}
	Set $\cM_k= \{i(k)\}$, with
	$$i(k) \in \argmax_{i \in [1  :  m]} \Sc{\g^{(i)}}{  \SSC(\x_k^{(i)}, -\nabla f(\x_k)^{(i)})-\x_{k}^{(i)} }.$$
\end{definition}

Finally, we have random sampling (see, e.g.,   \cite{lacoste2012block}). Here we randomly generate one index at each iteration with uniform probability distribution. The definition we have in this case is the following:

\begin{definition}[Random sampling] \label{d:random}
	Set $\cM_k = \{i(k)\}$, with $i(k)$ index chosen uniformly at random in~$[1 \! : \! m]$.
\end{definition}

\section{Convergence analysis} \label{s:crates}

In this section, we analyze the convergence properties of our algorithmic framework. In particular, we show that under a suitably defined angle condition on the blocks and a local KL condition on the objective function, we get,  for any block selection strategy used, a linear convergence rate. \damiano{The convergence analysis presented in this section extends the results given in \cite{rinaldi2020avoiding}
to the block coordinate setting, a demanding task which is by no means straightforward. Hence, some novel arguments are required for this extension, which are now introduced, and then described in detail in the appendix.} 

Our convergence framework makes use of the angle condition introduced in \cite{rinaldi2020unifying,rinaldi2020avoiding}. Such a condition ensures that the slope of the descent direction selected by the method is optimal up to a constant. We now recall this angle condition. For $\x \in \OM$ and $\g \in \rr^n$ we first define the directional slope lower bound as 
\begin{equation}
	\tx{DSB}_{\cA}(\OM, \x, \g) = \inf_{\d \in \cA(\x,\g)} \frac{\s{\g}{\d}}{\pi_\x(\g) \n{\d}} ,
\end{equation}
if $\x$ is not stationary for $-\g$, otherwise we set $\tx{DSB}_{\mathcal{A}}(\OM, \x, \g) = 1$. We then define the slope lower bound as
\begin{equation}
	\tx{SB}_{\cA}(\OM, P) = \inf_{\substack{\g \in \rr^n \\ \x \in P}} \tx{DSB}_{\cA}(\OM, \x, \g) = \inf_{\substack{\g: \pi_\x(\g) \neq 0 \\ \x \in P}} \tx{DSB}_{\cA}(\OM, \x, \g)  \, .
\end{equation}
We use $\tx{SB}_{\cA}(\OM)$ as a shorthand for $\tx{SB}_{\cA}(\OM, \OM)$, and say that the angle condition holds for the method $\cA$ if
\begin{equation} \label{eq:ang_cond}
	\tx{SB}_{\cA}(\OM) = \tau > 0 \, .
\end{equation}

\begin{remark}
AFW,  PFW and FDFW all satisfy the angle condition, when $\OM$ is a polytope. A detailed proof of this result is reported in \cite{rinaldi2020avoiding}, \damiano{together with some other examples of methods satisfying the angle condition for convex sets with smooth boundary described in \cite{rinaldi2020unifying}.}
\end{remark}

\newcommand{\qR}{q_R}
\newcommand{\qGS}{q_{GS}}
\newcommand{\qP}{q_P}

We now report the local KL condition  used to analyze the convergence of our algorithm. \damiano{The same condition was used  previously  as well~\cite[Assumption 2.1]{rinaldi2020avoiding}.}

\begin{assumption} \label{ass:KL}
	Given a stationary point $\x_* \in \OM$, there exists $\eta, \delta > 0$ such that for every 
 $\x \in B_{\delta}(\x_*)$ with $f(\x_*) < f(\x) < f(\x_*) + \eta$ we have
	\begin{equation} \label{hp:klP}
		\pi_\x(-\nabla f(\x)) \geq \ck [f(\x) - f(\x_*)]^{\frac{1}{2}} \, .
	\end{equation}
\end{assumption}

When dealing with convex programming problems, a H\"olderian error bound with exponent 2 on the solution set   implies  condition \eqref{hp:klP}, see  \cite[Corollary 6]{bolte2017error}. Therefore, our assumption holds when dealing with $\mu$-strongly convex functions (see, e.g., \cite{karimi2016linear}), \damiano{and in particular for the setting of the open question from \cite{osokin2016minding} discussed in the introduction}. It is however important to highlight that the error bound \eqref{hp:klP} holds in a variety of both convex and non-convex settings (see \cite{rinaldi2020avoiding} for a detailed discussion on this matter). An interesting example  for our analysis is the setting where $f$ is (non-convex) quadratic, i.e., $f(\x) = \x^{\top}\Qb\x + \b^{\top}\x$, and $\OM$ is a polytope.

We now report our main convergence result.

\begin{theorem} \label{th:bcrandom}
	Let Assumption~\ref{ass:KL} hold at $\x_*$. Let us  consider the sequence $\{\x_k\}$ generated by  Algorithm~\ref{algo6}. Assume that:
	\begin{itemize}
		\item the angle condition \eqref{eq:ang_cond} holds in every block for the same $\tau > 0$;
		\item the SSC procedure always terminates in a finite number of steps.
		\item $f(\x_*)$ is a minimum in the connected component of $ \{ \x\in \OM : f(\x) \leq f(\x_0)\} $ containing $\x_0$.
	\end{itemize}
	Then, there exists $\tilde{\delta} > 0$ such that, if $\x_0 \in B_{\tilde{\delta}}(\x_*)$:
	\begin{itemize}
		\item for the parallel block selection strategy, we have 
		\begin{equation} \label{th:qdecP}
			f(\x_k) - f(\x_*) \leq (\qP)^k [f(\x_0) - f(\x_*)]	\, ,
		\end{equation}
		and $\x_k \rightarrow \tilde{\x}_*$ with
		\begin{equation} \label{th:taillenP}
			\n{\x_k - \tilde{\x}_*} \leq \frac{\sqrt{2-2\qP}}{\sqrt{L}(1-\sqrt{\qP})}\, (\qP)^{\frac{k}{2}}[f(\x_0) - f(\tilde \x_*)] \, ,
		\end{equation}
		for 
		$$ \qP =   1 - \frac{\mu \tau^2}{4L(1 + \tau)^2}.$$
		\item for the GS block selection strategy, we have 
		\begin{equation} \label{th:qdecGS}
			f(\x_k) - f(\x_*) \leq (\qGS)^k [f(\x_0) - f(\x_*)]	\, ,
		\end{equation}
		and $\x_k \rightarrow \tilde{\x}_*$ with
		\begin{equation} \label{th:taillenGS}
			\n{\x_k - \tilde{\x}_*} \leq \frac{\sqrt{2-2\qGS}}{\sqrt{L}(1-\sqrt{\qGS})}\, (\qGS)^{\frac{k}{2}}[f(\x_0) - f(\tilde \x_*)] \, ,
		\end{equation}
		for
		$$\qGS =  1 - \frac{\mu \tau^2}{4mL(1 + \tau)^2} \, ,  $$ 
		
	 \item for the random block selection strategy we have, under the additional condition that 
  \begin{equation}\label{tech}\min\{ f(\x) : \n{\x - \x_*} = \delta\}> f(\x_*)\end{equation} 
	holds for some $ \delta > 0$,  that	
 \begin{equation} \label{th:qdecr}
			\mE[f(\x_k) - f(\x_*)] \leq (\qR)^k [f(\x_0) - f(\x_*)]	\, ,
		\end{equation}
		and $\x_k \rightarrow \tilde{\x}_*$ almost surely with
		\begin{equation} \label{th:taillenr}
			\mE[\n{\x_k - \tilde{\x}_*}] \leq \frac{\sqrt{2-2\qR}}{\sqrt{L}(1-\sqrt{\qR})}\, (\qR)^{\frac{k}{2}}[f(\x_0) - f(\tilde \x_*)]
		\end{equation}
		for $\qR = \qGS$. 
\end{itemize}
\end{theorem}

\damiano{This convergence result extends \cite[Theorem 4.2]{rinaldi2020avoiding} to our block coordinate setting.
However, since the SSC is here applied independently to different blocks, we cannot directly apply the results from \cite{rinaldi2020avoiding}. Instead, we combine the properties of the SSC applied in single blocks by exploiting the structure of the tangent cone for product domains:
	\begin{equation}
		T_{\OM}(\x) =  T_{\OM_{(1)}} (\x^{(1)})
  \times \cdots \times T_{\OM_{(m)}} (\x^{(m)})\, .
	\end{equation} 
This requires proving stronger properties for the sequence generated by the SSC than those presented in \cite{rinaldi2020avoiding}. The details with references to relevant results from \cite{rinaldi2020avoiding} can be found in the appendix. Finite termination of the SSC procedure is instead  directly ensured by the results proved in \cite{rinaldi2020unifying,rinaldi2020avoiding}, in particular for the AFW, PFW and FDFW applied on polytopes.}

\begin{remark}
	If the feasible set $\OM$ is a polytope and if we assume that the objective function $f$  satisfies condition~\eqref{hp:klP} on every point generated by the algorithm, with fixed $f(\x_*)$, then Algorithm~\ref{algo6} with AFW (PFW or FDFW) in the SSC converges at the rates given above. Condition~\eqref{hp:klP} holds in case of  $\mu$-strongly convex functions, and hence we have that in those cases our algorithm globally converges with the rates given in Theorem~\ref{th:bcrandom}. 
\end{remark}

\begin{remark}
	 Both the parallel and the GS strategy give the same rate with different constants. In particular, the constant ruling the GS case depends on the number of blocks used (the larger the number of blocks, the worse the rate) and is larger than the one we have for the parallel case.
\end{remark}

\begin{remark}
	 The random block selection strategy has the same rate as the GS strategy, but it is given in expectation.
In particular, the constant ruling the rate is the same as the GS one, hence  depends on the number of blocks used. Note that a further technical assumption~\eqref{tech} on $\x_*$  is needed in this case.
\end{remark}

	\section{Active set identification} \label{s:actid}
	We now report an active set identification result for our framework. We only focus on Algorithm~\ref{algo6} with AFW in the SSC and assume that strict complementarity holds and that the sets in the Cartesian product have a specific structure:

	\begin{equation}
		\OM =  \Delta^{n_1} \times ... \times \Delta^{n_m},
	\end{equation}
so	 that  the set $\OM_{(i)}$  is the $(n_i - 1)$-dimensional standard simplex $$\Delta^{n_i}=\{ \x\in \rr^{n_i}_+ :   \x^\top \e^{(i)}=1\}\, ,\quad i \in [1 \! : \! m]\, ,$$ for $\e \in \rr^{n}$ the vector with components all equal to 1.    We now report our main identification result.  A detailed proof is included in the appendix. 

	\begin{theorem} \label{th:actid}
		Under the above assumptions on $\OM$, let $\mathcal{A}^{(i)}$ be the AFW for $i \in [1 \! : \! m]$, and let strict complementarity conditions hold at $\x_* \in \OM$. 
		\begin{itemize}
			\item If $\{\x_k\}$ is generated by Algorithm \ref{algo6} with parallel selection, then there exists a neighborhood $U$ of $\x_*$ such that if $\x_k \in U$ then $\tx{supp}(\x_{k + 1}) = \tx{supp}(\x_*)$.
			\item If $\{\x_k\}$ is generated by Algorithm \ref{algo6} with randomized or GS selection, then there exists a neighborhood $U$ of $\x_*$ such that if $\x_k \in U$ then $\tx{supp}(\x_{k + 1}^{i(k)}) = \tx{supp}(\x_*^{i(k)})$.
		\end{itemize}
	\end{theorem}

  When the sequence generated by our algorithm converges to the point $\x_*$, it is then easy to see that the support of the iterate matches the final support of $\x_*$ for $k$ large enough.
\begin{corollary} \label{c:actid}
			Under the above assumptions on $\OM$, let $\mathcal{A}^{(i)}$ be the AFW for $i \in [1 \! : \! m]$, and let strict complementarity conditions hold at $\x_* \in \OM$. If $\x_k \rightarrow \x_*$ (almost surely), then for parallel and GS selection (for random sampling) we have $\tx{supp}(\x_k) = \tx{supp}(\x_*)$ for $k$ large enough.
\end{corollary}

This result has relevant practical implications, especially when handling sparse optimization problems. Since the algorithm iterates have a constant support when $k$ is large,  we 
can simply focus on the few support components and  forget about the others in this case. We hence can  exploit this by embedding sophisticated tools (like, e.g., caching strategies, second-order methods) in the algorithm, thus obtaining a significant speed up in the end. 

\section{Numerical results}\label{sec:numres}
We report here some preliminary numerical results for a non-convex quadratic optimization problem  referred to as Multi-StQP \cite{bomze2010multi} on a product of (here identical) simplices, that is
\begin{equation}\label{p:qsimplex}
		\min \left\{ \x^\top \Qb \x  : \x \in (\Delta^l)^m \right\}\, .
\end{equation}
The matrix $\Qb$ was generated in such a way that the solutions of problem \eqref{p:qsimplex} had components sparse but different from vertices. This is in fact the setting where FW variants have proved to be more effective \cite{bomze2022fast,rinaldi2020avoiding}. In order to obtain the desired property, we consider a perturbation of a stochastic StQP \cite{bomze2022two}. Given $\{\bar{\Qb}_i\}_{i \in [1 : m]}$ representing $m$ possible StQPs, with $\bar{\Qb}_i \in \rr^{l \times l}$ for $i \in [1 \! : \! m]$, the corresponding stochastic StQP with sample space $[1 \! : \! m]$ is given by
\begin{equation}\label{eq:pststqp}
	\max \left \{ \sum_{i = 1}^{m} p_i \y_i^\top \bar{\Qb}_i \y_i  : \y_i \in \Delta^l\quad \tx{for all } i \in [1 \! : \! m] \right \}\, .
\end{equation}
with $p_i$ probability of the StQP $i$. Equivalently, \eqref{eq:pststqp} is an instance of problem \eqref{p:qsimplex} with $\Qb = \bar{\Qb}$, for
\begin{equation}
	\bar{\Qb} = \begin{bmatrix} 
		-p_1\bar{\Qb}_1 & 0             & \cdots & 0            \\
		0            & -p_2\bar{\Qb}_2  & \cdots & 0            \\
		\vdots       & \vdots        & \ddots & \vdots       \\
		0            & 0             & \cdots & -p_m\bar{\Qb}_m
	\end{bmatrix} \, .
\end{equation}
 In our tests, we added to the stochastic StQP a perturbation coupling the blocks. More precisely, the matrix $\Qb$ was set equal to $\bar{\Qb} + \varepsilon \tilde{\Qb} $, for $\tilde{\Qb}$ a random matrix with standard Gaussian independent entries. The coefficient $\varepsilon$ was set equal to $\frac{1}{2m^2}$.  We set $\bar{\Qb}_i =\bar{\Ab}_i + \alpha \mathit{\Ib_l}$, for $\alpha = 0.5$ and $\bar{\Ab}_i$ the adjacency matrix of an Erd\H{o}s-R\'enyi random graph, where each couple of vertices has probability $p$ of being connected by an edge, independently from the other couples. Hence, for $i\in [1 \! : \! m]$ the problem
\begin{equation}
	\min \left\{- \y^\top\bar{\Qb}_i \y : \y \in \Delta^l \right\}
\end{equation}
is a regularized maximum-clique formulation, where each maximal clique corresponding to a unique strict local maximizer with support equal to its vertices, and conversely  (see \cite{bomze1999maximum} and references therein). The probability $p$ is set as follows
\begin{equation}
	p = \binom{l}{s}^{\frac{2}{s(s-1)}} \, ,
\end{equation}
for $s$ the nearest integer to $0.4l$, so that the expected number of cliques with size $ \approx 0.4l $ is 1, (see, e.g., \cite{alon2016probabilistic}). Notice that the perturbation term $\tilde{\Qb}$ ensures that problem \eqref{p:qsimplex} cannot be solved by optimizing each block separately. \\
We remark here that different ways to build large StQPs starting from smaller instances and preserving the structure of their solutions have been discussed in \cite{bomze2017complexity}. However, while the resulting problems decouple on the feasible set of the larger problem, they still decouple on the product of the feasible sets of the smaller instances, and for our purposes are equivalent to the block diagonal structure. 

 We tested four methods in total: AFW + SSC with parallel, GS and randomized updates (PAFW + SSC, GSAFW + SSC, BCAFW + SSC respectively), and FW with randomized updates (BCFW, coinciding with the block coordinate FW introduced in \cite{lacoste2012block}). Our tests focused on the local identification and on the convergence properties of our methods.

The code was written in {\tt Python} using the {\tt numpy} package, and the tests were performed on an Intel Core i9-12900KS CPU 3.40GHz, 32GB RAM. 
The codes relevant to the numerical tests are available at the following link: \\ \url{https://github.com/DamianoZeffiro/Projection-free-product-domain}.

\begin{figure}[h]
	\centering
		\begin{subfigure}[b]{0.25\textwidth}
		\includegraphics[width=\linewidth]{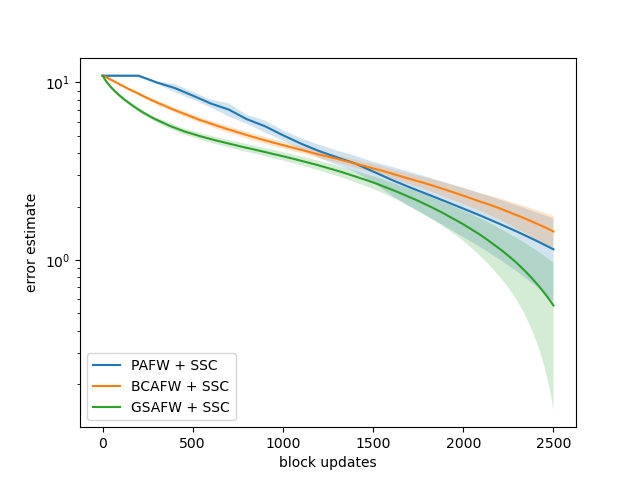}
	\end{subfigure}	
	\begin{subfigure}[b]{0.25\textwidth}
		\includegraphics[width=\linewidth]{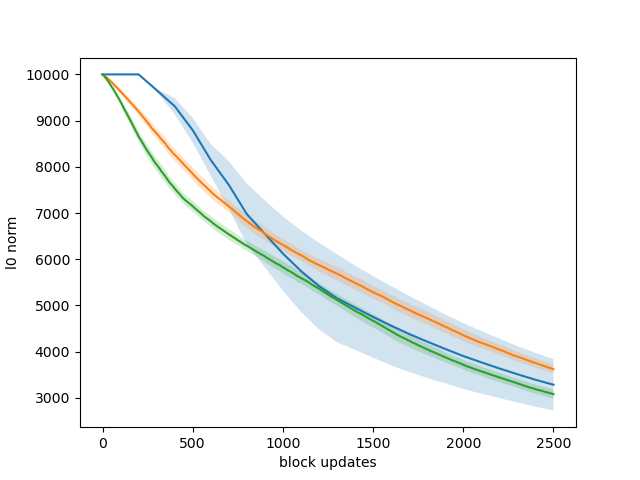}
	\end{subfigure}

	\begin{subfigure}[b]{0.25\textwidth}
		\includegraphics[width=\linewidth]{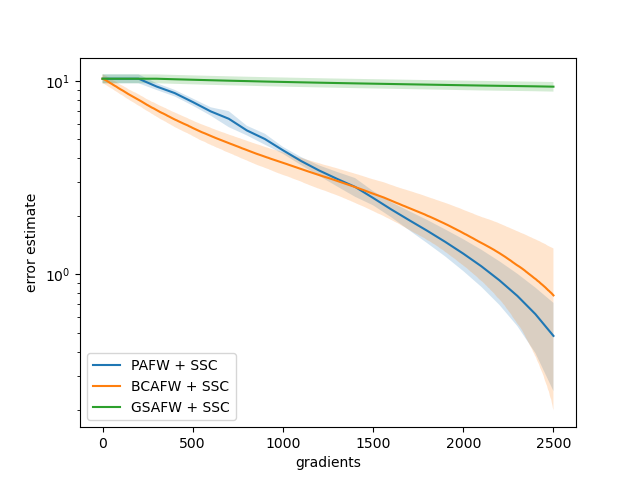}
	\end{subfigure}	
	\begin{subfigure}[b]{0.25\textwidth}
		\includegraphics[width=\linewidth]{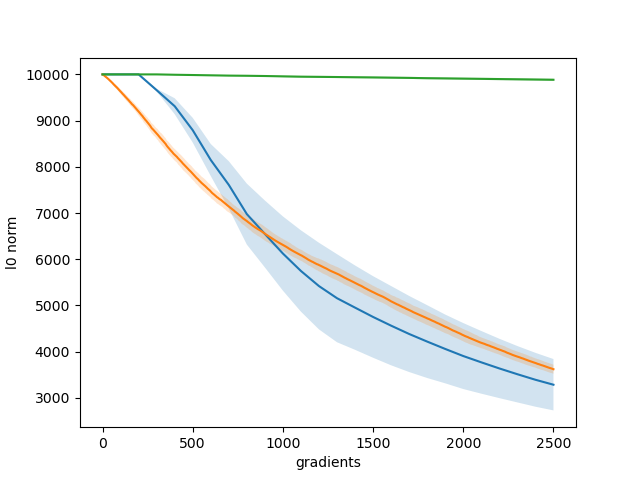}
	\end{subfigure}	
	\caption{Comparison using multistart between GSAFW + SSC, PAFW + SSC and BCAFW + SSC. $l = m = 100$.}
	\label{fig:t1snofw}
\end{figure}

 \subsection{Multistart}\label{s:multistart}
 We first considered a multistart approach, where the results are averaged across 20 runs, choosing 4 starting points for each of 5 random initializations of the objective.
 
 We measure both optimality gap  (error estimate) and sparsity (number of nonzero components, $\ell_0$ norm) of the iterates, reporting average and standard deviation in the plots. The  estimated global optimum used in the optimality gap is obtained by subtracting $10^{-5}$ from the best local solution found by the algorithms. We mostly consider the performance with respect to block gradient computations, with one gradient counted each time the SSC is performed in one of the blocks, as in previous works (see, e.g., \cite{lacoste2012block}). In some tests involving the GSAFW + SSC, we consider instead block updates, with one block update counted each time the algorithms modifies the current iterate in one of the blocks. \damiano{It is important to highlight that, 
 since at each block update the gradient is constant and only one linear minimization is required at the beginning of the SSC, the number of gradient computations for our algorithms also coincides with the number of linear minimizations on the blocks for the FW variants we consider.}
 \\ 
 We first compare PAFW  + SSC, BCAFW + SSC and GSAFW + SSC (Figure \ref{fig:t1snofw}). As expected, while GSAFW + SSC shows good performance with respect to block updates, it has a very poor performance with respect to block gradient computations, since at every iteration $m$ gradients must be computed to update a single block. We then compared PAFW + SSC, BCAFW + SSC and BCFW. The results (Figure \ref{fig:t1s}) clearly show that PAFW + SSC and BCAFW + SSC outperform BCFW. All these findings are consistent with the theoretical results described in Section \ref{s:asid}.

\begin{figure}[h]
	\centering
	\begin{subfigure}[b]{0.25\textwidth}
		\includegraphics[width=\linewidth]{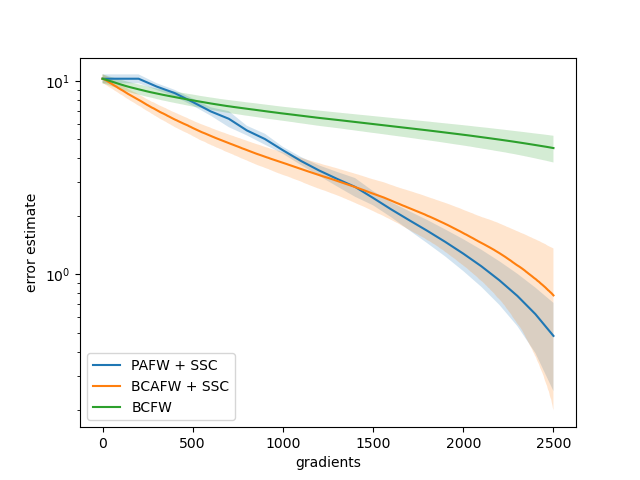}
	\end{subfigure}	
\begin{subfigure}[b]{0.25\textwidth}
	\includegraphics[width=\linewidth]{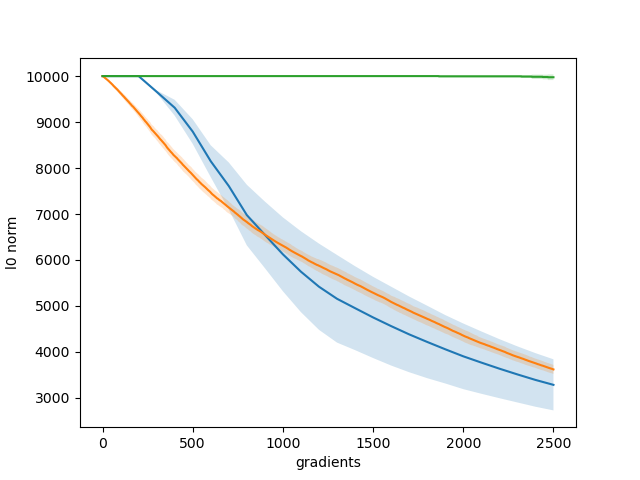}
\end{subfigure}	

	\begin{subfigure}[b]{0.25\textwidth}
	\includegraphics[width=\linewidth]{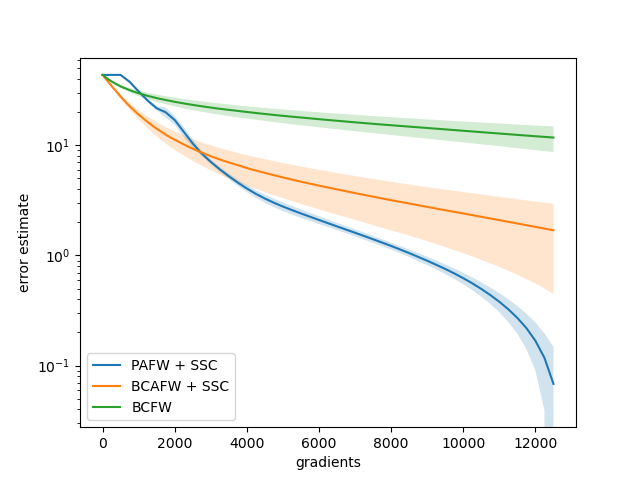}
\end{subfigure}	
\begin{subfigure}[b]{0.25\textwidth}
	\includegraphics[width=\linewidth]{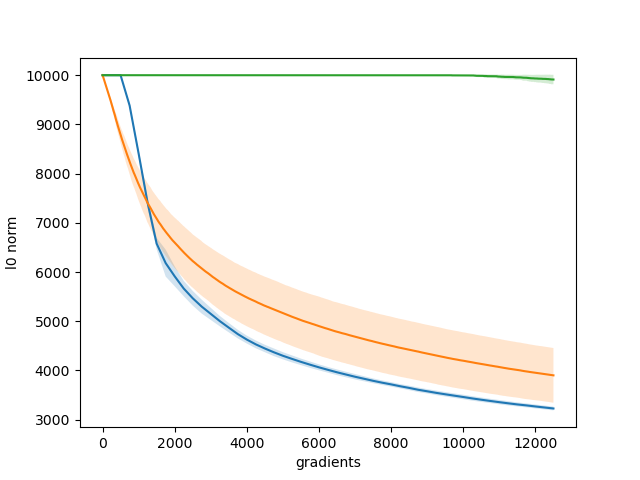}
\end{subfigure}	

	\begin{subfigure}[b]{0.25\textwidth}
	\includegraphics[width=\linewidth]{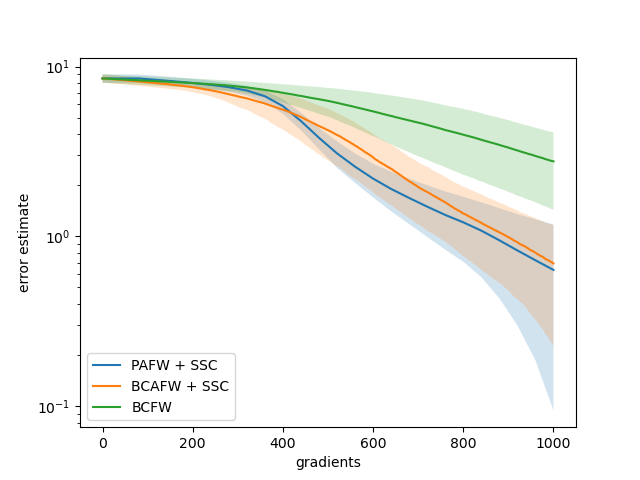}
\end{subfigure}	
\begin{subfigure}[b]{0.25\textwidth}
	\includegraphics[width=\linewidth]{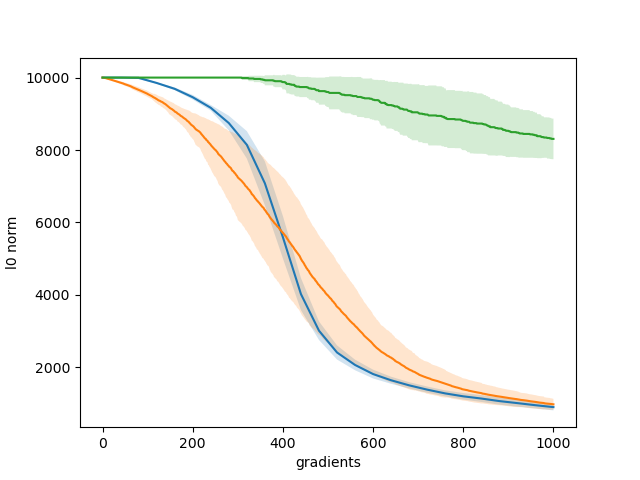}
\end{subfigure}	
	\caption{Comparison using multistart between BCFW, PAFW + SSC and BCAFW + SSC. $l = m = 100$ in the first row, $l=40$ and $m = 250$ in the second row, $l=250$ and $m=40$ in the third row.}
	\label{fig:t1s}
\end{figure}

\subsection{Monotonic basin hopping}
We then consider the monotonic basin hopping approach (see, e.g., \cite{leary2000global,locatelli2013global}) described in Algorithm~\ref{algo:7}. The method computes a local optimizer $\x_{*, i}$ close to the current iterate $\bar{\x}_i$ (Step 2). There $\mathcal{M}$ is a local optimization algorithm, and given as input $\mathcal{M}$ and $\bar{\x}_i$, the subroutine LO  returns the result of applying $\mathcal{M}$ starting from $\bar{\x}_i$, with a suitable stopping criterion which in our case is given by a limit on the number of gradient computations, set to $10m$. The sequence of best points found in the first $i$ iterations  $\{\bar{\x}_{*, i}\}$ is updated in Step 3, and in Step 5, $\bar{\x}_{i + 1}$ is chosen in a neighborhood of $\bar{\x}_{*, i}$. The neighborhood $B(\x, \gamma)$ for $\x \in \OM$ and $\gamma \in (0, 1]$ is given by 
\begin{equation}
B(\x, \gamma) = \{\x + \gamma(\y - \x) : \y \in \OM  \} \, .
\end{equation}
In the tests, we chose $\y$ uniformly at random in $\OM$ and set $\bar{\x}_{i + 1} = \bar{\x}_i + \gamma(\y - \bar{\x}_i)$, with $\gamma = 0.25$. 
\begin{algorithm}[h]
	\caption{Monotonic Basin Hopping Strategy}
	\begin{algorithmic}[1]
		\State $\bar{\x}_0 \in \OM$, $i_{\max} \in \mathbb{N}$, $\gamma \in [0, 1]$,  $i = 0$. Set $\bar{\x}_{*, - 1} = \bar{\x}_0$    
		\State Compute a local optimizer $\x_{*, i}=\tx{LO}(\mathcal{M}, \bar{\x}_i)$ 
            \State \textbf{If} $f(\x_{*, i}) < f(\bar{\x}_{*, i - 1})$ \textbf{then} set $\bar{\x}_{*, i} = \x_{*, i}$, \textbf{else} set $\bar{\x}_{*, i} = \bar{\x}_{*, i - 1}$.
       \State \textbf{If} $i = i_{\max}$, \textbf{then} STOP.      
		\State Randomly chose $\bar{\x}_{i + 1}$ in $B(\bar{\x}_{*, i}, \gamma)$.
            \State Set $i = i + 1$. Go to step 2.
	\end{algorithmic}
 \label{algo:7}
\end{algorithm} 
The methods we consider as subroutines in Step~2 are PAFW + SSC, BCAFW + SSC and BCFW. We set $i_{\max} = 9$, and perform 10 runs of Algorithm \ref{algo:7}, randomly initializing  the starting point. We plot once again average and standard deviation for $\{f(\bar{\x}_{*, i}) - \tilde{f}^*\}$ with $\tilde{f}^*$ estimating the global optimum (obtained by subtracting $10^{-1}$ from the best solution found by the methods). \\
The results again show that PAFW + SSC and BCAFW + SSC find better solutions than BCFW, with BCAFW + SSC outperforming PAFW + SSC in most instances if $l \le m$. 
\begin{figure}[h]
	\centering
	\begin{subfigure}[b]{0.25\textwidth}
		\includegraphics[width=\linewidth]{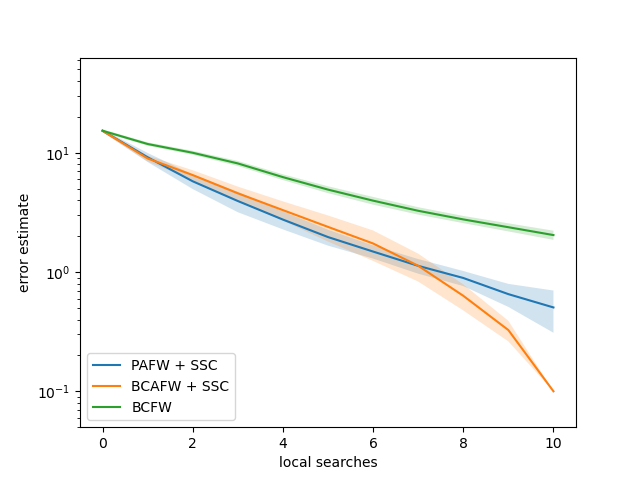}
	\end{subfigure}	
	\begin{subfigure}[b]{0.25\textwidth}
		\includegraphics[width=\linewidth]{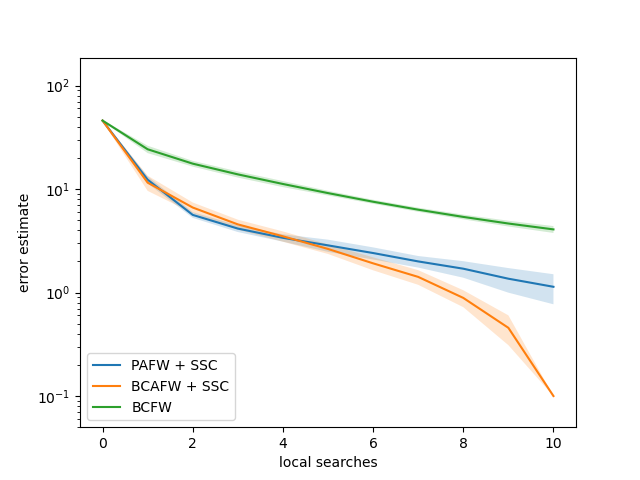}
	\end{subfigure}	
	\begin{subfigure}[b]{0.25\textwidth}
		\includegraphics[width=\linewidth]{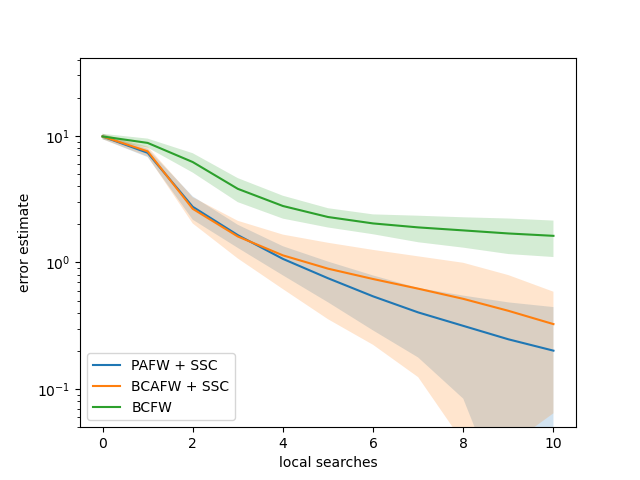}
	\end{subfigure}	
	\caption{Comparison using Monotonic Basin Hopping with BCFW, PAFW + SSC and BCAFW + SSC. From left to right: $l = m = 100$; $l=40$ and $m = 250$; $l=250$ and $m=40$.}
	\label{fig:t2s}
\end{figure}

\section{Conclusions}\label{sec:conc}

For a quite general optimization problem on product domains, we offer a seemingly new convergence theory, which ensures both convergence of objective values and (local) linear convergence of the iterates under widely accepted conditions, for block-coordinate FW variants. Convergence is global for $\mu$-strongly convex objectives, but we mainly focus on the non-convex case. In case of randomized selection of the blocks, all results are in expectation, and need a further technical assumption. As usual, constants and rates are specified in terms of the Lipschitz constant $L$ for the gradient map, the constant $\mu$ used in the local Kurdyka-\L ojasiewicz-condition,  and the parameter $\tau$ in the so-called angle condition.

The results are complemented by an active set identification result for a specific structure of the 
product domain and suitable choices of a projection-free strategy (FW-approach with away steps for the search direction): it is proved that our framework identifies the support of a solution in a finite number of iterations.

To the best of our knowledge, this is the first time that both a linear convergence
rate and an active set identification result are given for  (bad step-free) block-coordinate FW variants, in an effort to narrow the research gap observed in~\cite{osokin2016minding}.

In our preliminary experiments, numerical evidence clearly points out the advantages of our strategy to exploit structural knowledge. On randomly generated non-convex Multi-StQPs where easy instances were carefully avoided, our approach (AFW with parallel or randomized updates, both combined with the Short Step Chain strategy SSC) is dominating the block-coordinate FW method with randomized updates. 


We tested resilience of our reported observations by employing two experimental setups, pure multistart and monotonic basin hopping. The same effects seem to prevail.

Instance construction was motivated by a stochastic variant of the StQP, varying both domain dimension $l$ and number $m$ of possible scenarios. In case $l\le m$
there seems to be a slight edge towards the combination of AFW with randomized updates and SSC, compared to the parallel variant. This effect does not seem to happen with large $l$ in comparison to $m$, but would not change superiority over traditional block-coordinate FW methods.

\section{Appendix}
	\newcommand{\bSSC}{\overline{\textnormal{SSC}}}

	\subsection{Proofs}

	In the rest of this section, we always assume that the SSC terminates in a finite number of steps and that the angle condition holds. \\
	\damiano{The first lemma is related to the single block setting, and strengthens some of the properties proved for the SSC in \cite[Proposition 4.1]{rinaldi2020avoiding}}.
	\begin{lemma} \label{SSC:property}
		For a fixed $i \in [1 \! : \! m]$, let $\{\w_k\} = \{\x_k^{(i)}\}$, and let $\w_{k + 1} = \SSC(\w_k, \g)$. Then there exists $\tilde{\w}_k \in \{\y_j\}_{j=0}^T$ such that
		\begin{equation} \label{eq:Tcone}
			\n{ \w_{k +  1}-\w_k} \geq  \frac{\tau}{L} \n{\pi(T_{\OM_{(i)}}(\tilde{\w}_k), \g)}
		\end{equation}	      
		and
		\begin{align}
			& \n{\tilde{\w}_k - \w_k} \leq \n{\w_{k + 1} - \w_k} \label{eq:nmn} \, , \\
			& \Sc{\g}{\tilde{\w}_k - \w_k} \leq \Sc{\g}{\w_{k + 1} - \w_k} \label{eq:gwkwk}	 \, .
		\end{align}
		Furthermore, we have
		\begin{equation} \label{eq:normscg}
			L \n{\y - \w_k}^2 \leq \Sc{\g}{\y - \w_k} 
		\end{equation}
		for $y \in \{\w_{k + 1}, \tilde{\w}_k\}$.
	\end{lemma}
	\begin{proof}
		Let $\BB = \BB_{\frac{\n{\g}}{2L}}(\bar{\x} + \frac{\g}{2L})$ and let $T$ be such that $\w_{k +  1} = \y_T$. By \cite[(4.4)]{rinaldi2020avoiding} we have that \eqref{eq:normscg} holds for every $\z \in \BB$ (in place of $\y$), and therefore as desired for every 
  $$\y \in \{\w_{k + 1}, \tilde{\w}_k\} \subset \{\y_j : j\in [0 \! : \! T]\} \subset \BB\, .$$
		Let now $\tilde{p}_j = \n{\pi(T_{\OM_{(i)}}(\y_j), \g)}$. Notice that, if $\tilde{\w}_k = \y_l$, then
		\begin{equation}
			\frac{\tau}{L}\,	\n{\pi(T_{\OM_{(i)}}(\tilde{\w}_k), \g)} = \frac{\tau}{L}\, \tilde{p}_{l} \leq \frac{1}{L}\s{\g}{\hat{\d}_{l}} \, ,
		\end{equation}
		where the inequality follows from $\frac{\s{\g}{\hat{\d}_l}}{\tilde{p}_l} \geq \DSB_{\mathcal{A}}(\OM_{(i)}, \y_{l}, \g) \geq \SB_{\mathcal{A}}(\OM_{(i)}) = \tau$. Thus for proving \eqref{eq:Tcone}, in the rest of the proof it will be enough to prove
		\begin{equation} \label{eq:principle}
			\s{\g}{\hat{\d}_{l}} \leq L\n{\w_{k + 1}-\w_k} \, .
		\end{equation}  
		Furthermore, since by definition of the SSC, the scalar product $\Sc{\g}{\y_j}$ is increasing in $j$, we have
		$$\Sc{\g}{\tilde{\w}_k - \w_k} = \Sc{\g}{\y_l - \w_k} \leq \Sc{\g}{\y_T - \w_k} = \Sc{\g}{\w_{k + 1} - \w_k} \, .$$
		We distinguish four cases, according to how the SSC terminates. In the first two, we show we can choose the last step, $\tilde \w = \y_T$; in the third, the penultimate choice $\tilde \w = \y_{T-1}$ satisfies all conditions, and in the fourth case, an intermediate step is an appropriate choice. We abbreviate  $B_j = \BB_{\s{\g}{\hat{\d}_j}/L}(\w_k)$.\\
		\textbf{Case 1:} $T = 0$ or $\d_T = \oo$. Since there are no descent directions, $\w_{k +  1} = \y_T$ must be stationary for the gradient $-\g$.  Equivalently, $\tilde{p}_T = \n{\pi(T_{\OM_{(i)}}(\w_{k +  1}), \g)} = 0$. Finally, it is clear that if $T = 0$ then $\d_0 =\oo$, since $\y_0$ must be stationary for $-\g$. Thus taking $\tilde{\w}_k = \y_T$ the desired properties follow.  \\
		Before examining the remaining cases we remark that if the SSC terminates in Phase II, then $\alpha_{T- 1} = \beta_{T-1}$ must be maximal w.r.t. the conditions $\y_T \in B_{T-1}$ or $\y_T \in \bar{B}$. If $\alpha_{T-1} = 0$ then $\y_{T-1} = \y_T$, and in this case we cannot have $\y_{T-1} \in \partial \bar{B}$, otherwise the SSC would terminate in Phase II of the previous cycle. Therefore necessarily $\y_T = \y_{T-1} \in \tx{int}(B_{T-1})^c$ (Case 2). If $\beta_{T - 1} = \alpha_{T- 1} > 0$ we must have $\y_{T-1}\in \OM_{T-1} = B_{T-1} \cap \bar{B}$, and $\y_T \in \partial B_{T - 1}$ (Case 3) or $\y_T \in \partial \bar{B}$ (Case 4) respectively.  \\
		\textbf{Case 2:} $\y_{T-1} = \y_T \in \tx{int}(B_{T-1})^c$. We can rewrite the condition as
		\begin{equation} \label{trueT}
			\s{\g}{\hat{\d}_{T-1}} \leq L\n{\y_{T-1} - \w_k} = L \n{\y_T - \w_k} \, ,
		\end{equation}
		which is exactly \eqref{eq:principle}.
		Then $\tilde{\w}_k = \w_{k + 1} = \y_T$ satisfies the desired conditions. \\
		\textbf{Case 3:} $\y_T = \y_{T - 1} + \beta_{T - 1} \d_{T-1}$ and $\y_T \in \partial B_{T-1}$. Then from $\y_{T-1} \in B_{T-1}$ it follows
		\begin{equation} \label{c3y}
			L \n{\y_{T-1} - \w_k} \leq \s{\g}{\hat{\d}_{T-1}} \, ,
		\end{equation}
		and $\y_T \in \partial B_{T-1}$ implies
		\begin{equation}\label{t1T}
			\s{\g}{\hat{\d}_{T-1}} = L \n{\y_T - \w_k} \, ,
		\end{equation}
		which is \eqref{eq:principle} for $l = T - 1$.
		Combining \eqref{c3y} with \eqref{t1T} we also obtain
		\begin{equation} \label{yt1}
			L \n{\y_{T - 1} - \w_k} \leq L \n{\y_T - \w_k} \, ,
		\end{equation}
		so that in particular we can take $\tilde{\w}_k = \y_{T-1}$. \\
		\textbf{Case 4:} $\y_T = \y_{T - 1} + \beta_{T - 1} \d_{T-1}$ and $\y_T \in \partial \bar{B}$. \\
		The condition $\w_{k +  1} = \y_T \in \partial \bar{B}$ can be rewritten as
		\begin{equation} \label{case2c}
			L\n{\w_{k +  1} - \w_k}^2 - \s{\g}{\w_{k +  1} - \w_k} = 0 \, .
		\end{equation}
		For every $j \in [0 \! : \! T]$ we have
		\begin{equation} \label{eq:rec}
			\w_{k +  1} = \y_j + \sum_{i=j}^{T-1} \alpha_i \d_i \, .
		\end{equation}
		We now want to prove that for every $j \in [0 \! : \! T]$
		\begin{equation} \label{eq:claim2}
			\n{ \w_{k +  1} - \w_k} \geq \n{\y_j - \w_k} \, .
		\end{equation}
		Indeed, we have
		\begin{equation*}
			\begin{aligned}
				L\n{ \w_{k +  1} - \w_k}^2 = \s{\g}{\w_{k +  1} - \w_k} &= \s{\g}{\y_j - \w_k} + \sum_{i=j}^{T-1} \alpha_i \s{\g}{\d_i} \\ &\geq \s{\g}{\y_j - \w_k} \geq L\n{\y_j - \w_k}^2 \, ,
			\end{aligned}
		\end{equation*}
		where we used \eqref{case2c} in the first equality, \eqref{eq:rec} in the second, $\s{g}{\d_j} \geq 0$ for every $j$ in the first inequality and $\y_j \in \bar{B}$ in the second equality, which proves~\eqref{eq:claim2}. \\
		We also have
		\begin{equation} \label{eq:fracineq}
				\frac{\s{\g}{\w_{k +  1} - \w_k}}{\n{\w_{k +  1} - \w_k}}  = \frac{\s{\g}{\sum_{j=0}^{T-1}\alpha_j \d_j}}{\n{\sum_{j=0}^{T-1}\alpha_j \d_j}}  \geq
				\frac{\s{\g}{\sum_{j=0}^{T-1}\alpha_j \d_j}}{\sum_{j=0}^{T-1}\alpha_j \n{\d_j}} 
				\geq \min \left\{\frac{\s{\g}{\d_j}}{\n{\d_j}} :  j \in [0 \! : \! T-1] \right\} \, .
		\end{equation}
		Thus for $\tilde{T} \in \argmin \left\{ \frac{\s{\g}{\d_j}}{\n{\d_j}} : j \in [0 \! : \! T-1]\right\}$ we have
		\begin{equation} \label{tildeT}
			\s{g}{\hat{\d}_{\tilde{T}}} \leq 	\frac{\s{\g}{\w_{k +  1} - \w_k}}{\n{\w_{k +  1} - \w_k}} = L\n{\w_{k +  1} - \w_k} \, ,
		\end{equation}
		where we used \eqref{eq:fracineq} in the first inequality and \eqref{case2c} in the second (equality). \\
		In particular $\tilde{\w}_k = \y_{\tilde{T}}$ satisfies the desired properties, where $\n{\tilde{\w}_k - \w_k} \leq \n{\w_{k + 1} - \w_k}$ by \eqref{eq:claim2} and \eqref{eq:principle} holds by \eqref{tildeT}. \qed
	\end{proof}
	
	We denote by $\bSSC(\w_k, \g)$ a point $\tilde{\w}_k$ with the properties stated in the above lemma. It is also useful to define $U_0$ as the connected component of $\{ \x\in \OM : f(\x) \leq f(\x_0)\}$ containing $\x_0$.  \damiano{The next result shows how in our block coordinate setting the assumption of Theorem \ref{th:bcrandom} on $U_0$ allows us to retrieve a lower bound on the objective for points generated by the SSC. This lower bound is analogous to the lower bound required in \cite[Theorem 4.2]{rinaldi2020avoiding}.}
	\begin{lemma} \label{l:preliminary}
Let $\{\x_k\}$ be a sequence generated by Algorithm \ref{algo6}. Let $\bar{\x}_{k} = [\bSSC(\x_k^{(i)}, -\nabla f(\x_k)^{(i)})]_{i = 1}^m$. Assume also that $f(\x_*)$ is a minimum in $U_0$. Then, $\{f(\x_k)\}$ is decreasing, and for every $k$, $\{\x_k, \bar{\x}_k\} \subset U_0$, with $f(\y) \in [f(\x_*), f(\x_0)]$ for $\y \in  \{\x_k, \bar{\x}_k\}$.  
	\end{lemma}
\begin{proof}
	Let $U_k$ be the minimal connected component of $\{ \x\in \OM : f(\x) \leq f(\x_k)]$ containing $\x_k$, let $\g = - \nabla f(\x_k)$ and let $\bar{B}^{\OM}_k = \OM \cap \prod_i \bar{B}_{\frac{\n{\g^{(i)}}}{2L}}(\x_k^{(i)} + \frac{\g^{(i)}}{2L})$. For $\y \in \bar{B}^\OM_k$, we have $\x_k\in \bar{B}^{\OM}_k$ and
	\begin{equation}\label{eq:decp}
		\begin{aligned}
				& f(\y) \leq f(\x_k) + \Sc{\nabla f(\x_k)}{\y - \x_k} + \frac{L}{2}\n{\y - \x_k}^2  \\ 
				& = f(\x_k) + \sum_{i=1}^{m} \Sc{\nabla f(\x_k)^{(i)}}{\y^{(i)} - \x_k^{(i)}} + \frac{L}{2}\n{\y^{(i)} - \x_k^{(i)}}^2 \leq f(\x_k) - \frac{L}{2}\sum_{i=1}^{m}\n{\y^{(i)} - \x_k^{(i)}}^2 \\ 
				& =	f(\x_k) - \frac{L}{2}\n{\x_k - \y}^2
		\end{aligned}
	\end{equation}
where we used the standard Descent Lemma in the first inequality and the the second follows by definition of $\bar{B}^{\OM}_k$. From \eqref{eq:decp} it follows that $\{f(\x_k)\}$ is decreasing, and that $\bar{B}^{\OM}_k \subset \{ \x\in \OM : f(\x) \leq f(\x_k)\}$.  Furthermore, since $\bar{B}^{\OM}_k$ is connected and contains $\x_k$, the stronger inclusion $\bar{B}^{\OM}_k \subset U_k$ is also true. Thus $\{\x_{k + 1}, \bar{\x}_k\} \subset \bar{B}^{\OM}_k \subset U_k$, so that in particular $U_{k + 1} \subset U_k$ since $f(\x_{k + 1}) \leq f(\x_k)$, and by induction we can conclude $\{\x_{k + 1}, \bar{\x}_k\} \subset U_0$. Finally, $f(\y) \in [f(\x_*), f(\x_k)]$ for $\y \in \{\x_{k + 1}, \bar{\x}_k\}$, where the lower bound follows from the assumption that $f(\x_*)$ is a minimum in $U_0$, and the upper bound follows from \eqref{eq:decp}. \qed
\end{proof}
 \damiano{In the following lemma, the properties of the SSC proved in Lemma \ref{SSC:property} for single blocks are combined to obtain analogous properties on the whole product of blocks, and the KL condition is then used to lower bound suitable improvement measures with an optimality gap for the objective. We would like to highlight that, unlike the single block case, this optimality gap is measured with respect to an auxiliary point which is not necessarily among those generated by the algorithm. Proof of the linear convergence rate  hence requires proper handling in this case.}
		\begin{lemma} \label{l:intermediate}
		Let $\{\x_k\}$ be a sequence generated by Algorithm \ref{algo6}, and assume that the angle condition holds for the method $\cA^{(i)}$ with the same $\tau $, for all $i\in [1 \! : \! m]$. Let $\bar{\x}_{k} = [\bSSC(\x_k^{(i)}, -\nabla f(\x_k)^{(i)})]_{i = 1}^m$ and $\tilde{\x}_{k + 1} = [\SSC(\x_k^{(i)}, -\nabla f(\x_k)^{(i)})]_{i = 1}^m$. If \eqref{hp:klP} holds at $\bar{\x}_k$, we then have, abbreviating $\g = -\nabla f(\x_k)$:
		\begin{align}
			&     \n{\tilde{\x}_{k + 1} - \x_k}^2 \geq \frac{\tau^2}{2(1 + \tau^2)L^2}\n{\pi(T_{\OM}(\bar{\x}_k), -\nabla f(\bar{\x}_k))}^2\,  \,\geq \frac{\tau^2 \mu}{L^2(1 + \tau^2)}[f(\bar{\x}_k) -  f^*] \label{eq:fbound2.0}\, , \\
			&      \frac{1}{2}\Sc{\g}{\tilde{\x}_{k + 1} - \x_k} \geq \frac{1}{3}[f(\x_k) - f(\bar{\x}_k)]  \label{eq:fbound2.1} \, .
		\end{align}
	\end{lemma}

\begin{proof}
	Let $\bar{\g} = -\nabla f(\bar{\x}_k)$, $\bar{q}_{(i)} = \n{\pi(T_{\OM_{(i)}}(\bx_k^{(i)}), \bar{\g}^{(i)})}$, and $q_{(i)} = \n{\pi(T_{\OM_{(i)}}(\bar{\x}_k^{(i)}), \g^{(i)})}$.
	Observe that by the Lipschitz continuity of the gradient, we have the inequality
	\begin{equation}\label{eq:lipcon}
		\bar{q}_{(i)} \leq q_{(i)} + L\n{\bx_{k}^{(i)} - \x_k^{(i)}}
	\end{equation}
	and thus 
	\begin{equation}\label{eq:lipcon2}
		\bar{q}_{(i)}^2 \leq  2 q_{(i)}^2 + 2L^2\n{\bx_{k}^{(i)} - \x_k^{(i)}}^2 \leq \frac{2L^2(1 + \tau^2)}{\tau^2}\n{\tilde{\x}_{k + 1}^{(i)} - \x_k^{(i)}}^2 \, ,
	\end{equation}
	where we applied Jensen's inequality to \eqref{eq:lipcon} in the first inequality, and \eqref{eq:Tcone} together with \eqref{eq:nmn} in the second inequality. \\
	Thus we can write
	\begin{equation}
		\begin{aligned}
	\n{ \tilde{\x}_{k + 1}-\x_k }^2 		 
	&= \sum_{i=1}^m \n{ \tilde{\x}_{k + 1}^{(i)}-\x_k^{(i)}}^2  \geq \frac{\tau^2}{2L^2(1 + \tau^2)} \sum_{i=1}^{m}  \bar{q}_{(i)}^2 \\
			& = \frac{\tau^2}{2L^2(1 + \tau^2)} \n{\pi(T_{\OM}(\bar{\x}_k), -\nabla f(\bar{\x}_k))}^2 \,  \,\geq \frac{\tau^2 \mu}{L^2(1 + \tau^2)}[f(\bar{\x}_k) -  f^*] \, ,
		\end{aligned}
	\end{equation}
	where we used \eqref{eq:lipcon2} in the first inequality and the KL property in the second.	This proves \eqref{eq:fbound2.0}.\\ 
	Using the standard Descent Lemma, we can give the upper bound
	\begin{equation} \label{eq:fxkfbk}
		f(\x_k) - f(\bx_k) \leq \Sc{\g}{\bx_k - \x_k} + {\textstyle\frac{L}{2}}\,\n{\bx_k - \x_k}^2 \leq {\textstyle\frac{3}{2}}\,\Sc{\g}{\bx_k - \x_k} = {\textstyle\frac{3}{2}} \sum_{i=1}^m \Sc{\g^{(i)}}{\bx_k^{(i)} - \x_k^{(i)}} \, ,
	\end{equation}
	where we used \eqref{eq:normscg} in the second inequality.  We can finally prove \eqref{eq:fbound2.1}:
	\begin{equation}
 {\textstyle\frac{1}{2}} \Sc{\g}{\tilde{\x}_{k + 1} - \x_k} \geq {\textstyle\frac{1}{2}} \Sc{\g}{\bar{\x}_k - \x_k} = 	{\textstyle\frac{1}{2}} \sum_{i=1}^m \Sc{\g^{(i)}}{\bar{\x}_k^{(i)} - \x_k^{(i)}} 
 \geq {\textstyle\frac{1}{3}} [f(\x_k) - f(\bar{\x}_k)] \, ,
	\end{equation}
	where we used~\eqref{eq:gwkwk} in the first inequality and \eqref{eq:fxkfbk} in the second one. \qed
\end{proof}
	
\damiano{The next result,  which directly follows from the previous lemma, explicitly lower bounds the improvement on the objective with the optimality gap introduced above.}
	\begin{lemma} \label{l:Epineq}
		Let $\{\x_k\}$ be a sequence generated by Algorithm \ref{algo6}, and assume that the angle condition holds for the method $\cA^{(i)}$ with the same $\tau $, for all $i\in [1 \! : \! m]$. Let $\bar{\x}_{k} = (\bSSC(\x_k^{(i)}, -\nabla f(\x_k)^{(i)}))_{i = 1}^m$. Then, if the KL property \eqref{hp:klP} holds at $\bar{\x}_k$,
	for parallel updates
	\begin{equation}\label{parmal}
f(\x_k)-		f(\x_{k + 1})  \geq \frac{\tau^2 \mu}{2L(1 + \tau^2)}\, (f(\bar{\x}_k) - f^*) \, ,
	\end{equation}
for GS updates
\begin{equation}
f(\x_k)-		f(\x_{k + 1})  \geq \frac{\tau^2 \mu}{2L(1 + \tau^2)}\,\frac{1}{m}\, (f(\bar{\x}_k) - f^*) \, ,
\end{equation}
and for random updates
\begin{equation}
	\mE[f(\x_k)-		f(\x_{k + 1}) ] \geq \frac{\tau^2 \mu}{2L(1 + \tau^2)}\, \frac{1}{m}\,\mE[f(\bar{\x}_k) - f^*] \, . \label{eq:fbound2}
\end{equation}
	\end{lemma}
\begin{proof}
	We first prove the inequality for parallel updates. We have
	\begin{equation}\label{eq:parma}
		\begin{aligned}
	&		f(\x_k) - f(\x_{k + 1}) \geq \frac{L}{2} \,\n{\x_k - \x_{k + 1}}^2 \geq \frac{L}{2}\,\frac{\tau^2 }{2L^2(1 + \tau^2)}\,\n{\pi(T_{\OM}(\bar{\x}_k), -\nabla f(\bar{\x}_k))}^2 \\ 
	&	= \frac{\tau^2 }{4L(1 + \tau^2)}\n{\pi(T_{\OM}(\bar{\x}_k), -\nabla f(\bar{\x}_k))}^2 \,  \,\geq \frac{\tau^2 \mu}{2L^2(1 + \tau^2)}\,[f(\bar{\x}_k) -  f^*] \, , 
		\end{aligned}
	\end{equation}
where the first inequality follows from \eqref{eq:decp}, the second inequality
by \eqref{eq:fbound2.0} where with the notation introduced in Lemma \ref{l:intermediate} we have by definition $\x_{k + 1} = \tilde{\x}_{k + 1}$. For GS updates, we have
	\begin{equation}\label{eq:GSgx}
\begin{aligned}
	& f(\x_k) - f(\x_{k + 1}) \geq \Sc{g}{\x_{k + 1} - \x_k} - \frac{L}{2} \n{\x_{k + 1} - \x_k}^2 \geq \frac{1}{2} \Sc{\g}{\x_{k + 1} - \x_k} \\
	& = \frac{1}{2} \max_{i \in [1 \! : \! m]} \Sc{\g^{(i)}}{\tilde{\x}_{k + 1}^{(i)} - \x_k} \geq \frac{1}{2m} \Sc{\g}{\tilde{\x}_{k + 1} - \x_k} \geq \frac{L}{2m} \n{\tilde{\x}_{k + 1} - \x_k}^2 \\ 
	& \geq  \frac{\tau^2 \mu}{2mL(1 + \tau^2)}\,[f(\bar{\x}_k) - f^*] \,   ,
\end{aligned}
	\end{equation}
where in the first inequality we used the standard Descent Lemma, \eqref{eq:normscg} in the second inequality; the equality follows by definition of GS updates, in the fourth inequality we applied again \eqref{eq:normscg}, and \eqref{eq:fbound2.0} in the last one. \\
Finally, for random updates we have, denoting as ${i(k) = j}$ the event that the index chosen at the step $k$ is $j$:

\begin{equation}
	\begin{aligned}
		& \mE[f(\x_k) - f(\x_{k + 1})] \geq \frac{L}{2} \mE[\n{\x_{k + 1} - \x_k}^2] =  \frac{L}{2} \sum_{j = 1}^m \Prob({i(k) = j}) \mE[\n{\tilde{\x}_{k + 1}^{(j)} - \x_k^{(j)}}^2]  \\
		 & =  \frac{L}{2m} \sum_{j = 1}^m \mE[\n{\tilde{\x}_{k + 1}^{(j)} - \x_k^{(j)}}^2] = 
		 \frac{L}{2m}\mE[\n{\tilde{\x}_{k + 1} - \x_k}^2] \geq \frac{\tau^2 \mu}{2L(1 + \tau^2)}\mE[f(\bar{\x}_k) -  f^*] \, ,
	\end{aligned}
\end{equation}
where the first inequality follows from \eqref{eq:decp}, we used $\Prob(\{i(k) = j \}) = \frac 1m$ in the second equality and \eqref{eq:fbound2.0} in the last inequality. \qed
\end{proof}

 \damiano{In the next two lemmas, we relate the improvement measured with respect to the auxiliary point to the true improvement of the objective, and thus manage to extend the linear convergence rate in \cite[Lemma 4.3]{rinaldi2020avoiding} to the block coordinate setting.}
	\begin{lemma} \label{l:Epineq2}
	Let $\{\x_k\}$ be a sequence generated by Algorithm \ref{algo6}, and assume that the angle condition holds for the method $\cA^{(i)}$ with the same $\tau $, for all $i\in [1 \! : \! m]$. Let $\bar{\x}_{k} = (\bSSC(\x_k^{(i)}, -\nabla f(\x_k)^{(i)}))_{i = 1}^m$. Then
	for parallel updates
	\begin{equation}
	f(\x_k) - f(\x_{k + 1}) \geq \frac{1}{3}[f(\x_k) - f(\bar{\x}_k)] \label{eq:fboundP}\, ,
	\end{equation}
	for GS updates
	\begin{equation}
			f(\x_k) - f(\x_{k + 1}) \geq \frac{1}{3m}[f(\x_k) - f(\bar{\x}_k)] \label{eq:fboundGS} \, ,
	\end{equation}
	and for random updates
	\begin{equation}
		\mathbb{E}[f(\x_k) - f(\x_{k + 1})] \geq \frac{1}{3m}\mE[f(\x_k) - f(\bar{\x}_k)] \label{eq:fboundR}
	\end{equation}
\end{lemma}
\begin{proof}
For parallel updates, we have
\begin{equation}
	\begin{aligned}
	& f(\x_k) - f(\x_{k + 1})  \geq \Sc{\g}{\x_{k + 1}- \x_k} - \frac{L}{2}\n{\x_{k + 1} - \x_k}^2 \geq \frac{1}{2}\Sc{\g}{\x_{k + 1} - \x_k}  \\
	& = \frac{1}{2}\Sc{\g}{\tilde{\x}_{k + 1} - \x_k} \geq \frac{1}{3}[f(\x_k) - f(\bar{\x}_k)] \, ,		
	\end{aligned}
\end{equation}
where we have used the standard descent Lemma in the first inequality, \eqref{eq:normscg} in the second inequality, and \eqref{eq:fbound2.1} in the last inequality. \\
The proof follows analogously for GS updates, after noticing
\begin{equation}
\Sc{\g}{\x_{k + 1} - \x_k} \geq \frac{1}{m} \Sc{\g}{\tilde{\x}_{k + 1} - \x_k} \, , 
\end{equation}
as showed in \eqref{eq:GSgx}, and for random updates, using
\begin{equation}
\mE[\Sc{\g}{\x_{k + 1} - \x_k}] = \frac{1}{m} \mE[\Sc{\g}{\tilde{\x}_{k + 1} - \x_k}] \, ,
\end{equation}
respectively. \qed
\end{proof}
\begin{lemma}\label{l:crate_main}
 		Let $\{\x_k\}$ be a sequence generated by Algorithm \ref{algo6}, and assume that the angle condition holds for the method $\cA^{(i)}$ with the same $\tau $, for all $i\in [1 \! : \! m]$. Then, if the KL property \eqref{hp:klP} holds at $\x_k$, for parallel updates
 	\begin{equation}
 		f(\x_{k + 1}) - f^* \leq \qP (f(\x_k) - f^*) \, ,
 	\end{equation}
 	for GS updates
 	\begin{equation}
 		f(\x_{k + 1}) - f^* \leq \qGS (f(\x_k) - f^*)  \, ,
 	\end{equation}
 	and for random updates
 	\begin{equation}
 		\mE[f(\x_{k + 1}) - f^*] \leq \qR \mE[f(\x_k) - f^*] \, .
 	\end{equation}
\end{lemma}
\begin{proof}
First observe that since $\tau \in [0, 1]$ and $\mu \leq L$  we have
\begin{equation}
	\frac{\tau^2 \mu}{2L(1 + \tau^2)} \,\leq\, \frac {\mu}{3L} \,\leq\,	\frac{1}{3} \, . 
\end{equation}
Then, 
combining~\eqref{parmal} and~\eqref{eq:fboundP}, we can write
\begin{equation}
	\begin{aligned}
		& 2[f(\x_k) - f(\x_{k + 1})] \geq \frac{1}{3}[f(\x_k) - f(\bar{\x}_k)] + \frac{\tau^2 \mu}{2L(1 + \tau^2)}\, [f(\bar{\x}_k) - f^*] \\ 
		& \geq \frac{\tau^2 \mu}{2L(1 + \tau^2)}\, [f(\x_k) - f(\bar{\x}_k)] + \frac{\tau^2 \mu}{2L(1 + \tau^2)}\,[f(\bar{\x}_k) - f^*] \\
		& = \frac{\tau^2 \mu}{2L(1 + \tau^2)}\, [f(\x_k) - f^*)] \, ,	
	\end{aligned}
\end{equation}
 and rearranging
\begin{equation}
	f(\x_{k + 1}) - f^*  \leq \qP [f(\x_k) - f^*] \, .
\end{equation}
The thesis follows for GS and random updates analogously. \qed
\end{proof}
\newcommand{\tdelta}{\tilde{\delta}}
\damiano{Given the previous results for the block coordinate setting, the remaining part of the proof is a straightforward adaptation of arguments used in the proof of \cite[Theorem 4.2]{rinaldi2020avoiding}}.
	\begin{proof}[of Theorem \ref{th:bcrandom}]
We need to prove that the KL property \eqref{hp:klP} holds in $\{\bar{\x}_k\}$. The bounds on $f(\x_k) - f(\x_*)$ then follow immediately by induction from Lemma \ref{l:crate_main}, and in turn the bounds on $\n{\x_k - \x_*}$ follow as in the proof of \cite[Lemma 4.3]{rinaldi2020avoiding}.  \\
For random updates, we can take $\tilde{\delta} < \delta$ small enough so that $f(\x_0) < f(\x_*) + \eta$. Then by construction the KL property \eqref{hp:klP} holds in $U_0$, and since $\{\bar{\x}_k\}$ is contained in $U_0$ by Lemma \ref{l:preliminary}, \eqref{hp:klP} holds in particular in $\{\bar{\x}_k\}$. \\
For parallel updates, thanks to Lemma \ref{l:preliminary} we have that $\{f(\x_k)\}$ is decreasing and $f(\bar{\x}_k), f(\x_k) \geq f(\x_*)$. It can then be proved with an argument analogous to the proof of \cite[Theorem 4.2]{rinaldi2020avoiding} that for $\tilde{\delta}$ small enough, \eqref{hp:klP} holds in $\{\bar{\x}_k\}$. We include the argument here for completeness. Let $f_k = f(\x_k) - f(\x_*)$, and let $\tilde{\delta} < \delta/2$ defined as in the proof of \cite[Theorem 4.2]{rinaldi2020avoiding}  so that
\begin{equation}\label{eq:tdelta}
	\tdelta < \frac{\delta}{2} < \delta - \frac{\sqrt{2f_0(1 - q)}}{L(1 - \sqrt{q})} - \sqrt{\frac{2}{L}}\sqrt{f_0} \, ,
\end{equation}	
with $q=\qP$ here. We now want to prove 
$\bigcup _{[0:k-1]}\{\x_i,\bar{\x}_i\} \cup\{\x_k\}\subset B_{\delta}(\x_*)$ 
for every 
$k \in \mathbb{N}$, by induction on $k$. Notice that $\x_0 \in B_{\delta}(\x_*)$ by construction. To start with the inductive step, 
\begin{equation} \label{eq:firsttail}
	\sum_{i = 0}^{k - 1} \n{\x_i - \x_{i + 1}} \leq \sqrt{\frac{2}{L}} \sum_{i = 0}^{k - 1} \sqrt{f_{i} - f_{i + 1}} \leq \frac{\sqrt{2f_0(1 - q)}}{\sqrt{L}(1 - \sqrt{q})}
\end{equation}
where we used \eqref{eq:decp} in the first inequality, and the second can be derived from \cite[Lemma 8.1]{rinaldi2020avoiding} as in the proof of \cite[Theorem 4.2]{rinaldi2020avoiding}. But then
\begin{equation}\label{eq:inball1}
	\begin{aligned}
		& \n{\x_{k + 1} - \x_*} \leq \n{\x_0 - \x_*} + \left( \sum_{i = 0}^{k - 1} \n{\x_i - \x_{i + 1}} \right) + \n{\x_k - \x_{k + 1}} \\ 
		& \leq \tdelta + \frac{\sqrt{2f_0(1 - q)}}{L(1 - \sqrt{q})} + \sqrt{\frac{2}{L}}\sqrt{f_k - f_{k + 1}} \\
		& <	\tdelta + \frac{\sqrt{2f_0(1 - q)}}{L(1 - \sqrt{q})} + \sqrt{\frac{2}{L}}\sqrt{f_k} < \delta \, ,	
	\end{aligned}
\end{equation}
where we used \eqref{eq:firsttail} together with \eqref{eq:decp} in the second inequality, $f_{k + 1} = f(\x_{k + 1}) - f(\x_*) \geq 0$ in the third inequality, and \eqref{eq:tdelta} together with $f_0 \geq f_{k}$ in the last inequality. \\
We now have
\begin{equation}\label{eq:inball2}
	\begin{aligned}
		& \n{\tilde{\x}_{k} - \x_*} \leq \n{\x_0 - \x_*} + \left( \sum_{i = 0}^{k - 1} \n{\x_i - \x_{i + 1}} \right) + \n{\x_k - \tilde{\x}_{k}} \\ 
		& \leq \n{\x_0 - \x_*} + \left( \sum_{i = 0}^{k - 1} \n{\x_i - \x_{i + 1}} \right) + \n{\x_k - \x_{k + 1}} < \delta \, ,	
	\end{aligned}
\end{equation}
where we used $\n{\tilde{\x}_k - \x_k} \leq \n{\x_{k + 1} - \x_k}$ in the second inequality and the last inequality follows as in \eqref{eq:inball1}. Thus $\tilde{\x}_k \in B_{\delta}(\x_*)$ as well, and the induction is complete. For GS updates the proof that $\{\tilde{\x}_k\} \subset B_{\delta}(\x_*)$ is analogous. \qed
	\end{proof}

\subsection{An active set identification criterion} \label{s:asid}
We prove in this section Theorem \ref{th:actid}, proposing a general active set identification criterion for Algorithm \ref{algo6} in the special case where the feasible set $\OM$ is the product of simplices. With the notation introduced in Section \ref{s:actid}, let $\OM_* = \{\x \in \OM :\supp(\x) = \supp(\x_*) \}$ and $S_* = \{\x \in \rr^n :\supp(\x) = \supp(\x_*)\}$ be the subset of points in $\OM$ and the subspace of directions with the same support of $\x_*$ respectively. 
\begin{definition} \label{def:asrel}
	We say that the method $\bA$ has active set related directions in $\x_*$ if it can do a bounded number of consecutive maximal steps, and if for some neighborhood $V$ of $\x_*$, $\x \rightarrow \x_*$, $\g \rightarrow -\nabla f(\x_*)$ and $\d \in \bA(\x, \g)$:
	\begin{itemize}
		\item if $\x \in \OM_*$ then $\d \in S_*$ with $\alpha_{\max}(\x, \hat{\d}) = \Theta(1) $,
		\item if $\x \in \OM \sm \OM_*$ then $\Sc{\g}{\hat{\d}} = \Theta(1)$.
	\end{itemize}
\end{definition}
\begin{lemma}
	Under the assumptions of Definition \ref{def:asrel}:
	\begin{itemize}
		\item if $\x \in \OM \sm \OM_*$ we have $\alpha_{\max}(\x, \hat{\d}) = o(1)$, 
		\item if $\x \in \OM_*$, then $\Sc{\g}{\hat{\d}} = o(1)$.
	\end{itemize}
\end{lemma}
\begin{proof}
	Notice that
	\begin{equation}
		\begin{aligned}
			& 0 \leq \alpha_{\max}(\x, \hat{\d})\Sc{\g}{\hat{\d}} = \Sc{\g}{(\x + \alpha_{\max}(\x, \hat{\d})\hat{\d}) - \x} \\
			& = \Sc{-\nabla f(\x)}{(\x + \alpha_{\max}(\x, \hat{\d})\hat{\d}) - \x} + o(1) \leq \Sc{-\nabla f(\x)}{\x_* - \x} + o(1) = o(1) \, .	
		\end{aligned}
	\end{equation}	
	Thus
	\begin{equation}
		\alpha_{\max}(\x, \hat{\d})  \leq \frac{o(1)}{\Sc{\g}{\hat{\d}}} = o(1) \, ,
	\end{equation}
	where in the equality we used $\Sc{\g}{\hat{\d}} = \Theta(1)$ by assumption. This proves the first part of the claim. As for the second part, we have
	\begin{equation}
		\Sc{\g}{\hat{\d}} = \Sc{-\nabla f(\x_*)}{\hat{\d}} + o(1) = o(1) \, ,
	\end{equation} 
	where we used $\Sc{-\nabla f(\x_*)}{\hat{\d}} = 0$ in the second equality, guaranteed by stationarity conditions since $\d \in S_*$. \qed
\end{proof}
\newcommand{\bgl}{\bar{B}}
\begin{proposition} \label{p:cactid}
	Let Algorithm \ref{algo6} be applied to a method with active set related directions in $\x_*$ as in Definition \ref{def:asrel}. Then there is a neighborhood $U$ of $\x_*$ such that if $\x_k \in U$ then $\supp(\x_{k + 1}) = \supp(\x_*)$.
\end{proposition}	
\begin{proof}
	Let $\{\y_i : i\in [0 \! : \! j]\}$ be the set of points generated by $\tx{SSC}(\x_k, -\nabla f(\x_k))$, $\bar{T}$ the upper bound on the number of consecutive maximal steps, so that $T \leq \bar{T} + 1$, and let $\bgl$ and $B_j$ as in the proof of Lemma~\ref{SSC:property}. We assume without loss of generality that $\n{\d_j} = 1$ for $j \in [0 \! : \!  T]$. \\ 
	We will show that for $\x_k$ sufficiently close to $\x_*$ certain inequalities, namely \eqref{eq:pscal}, \eqref{eq:pscal2} and \eqref{pscal3} are satisfied, allowing us to deduce the identification property. Let 
 $$T^* = \max \{j \in [0 \! : \! T] :\{\y_i : i\in [0 \! : \! j]\}  \subset \OM \sm \OM_* \}$$ whenever $\y_0 \notin \OM_*$, and $T^* = -1$ otherwise.  
	We first claim $\bar{T} \geq T^* + 1$. This is clear by the definition of $T^*$ if $T^* = -1$. Otherwise, for $j \in [0 \! : \! T^*]$ let $\tilde{\y}_{j + 1} = \y_{j} + \alpha_{\max}^{(j)}\d_j$. We now show $\tilde{\y}_{j + 1} \in \OM_{j}$. First, we check $\tilde{\y}_{j + 1} \in \tx{int}(\bgl)$. On one hand we have
	\begin{equation} \label{eq:Lnorm}
		\begin{aligned}
		& L\n{\tilde{\y}_{j + 1} - \y_0}^2 = L\n{\alpha_{\max}^{(j)}\d_j + \sum_{i=0}^{j - 1} \alpha_i \d_i}^2 \leq L \left( \sum_{i=0}^{j - 1}  \alpha_{\max}^{(i)}  \n{\d_i}\right)^2\\
		& = L \left( \sum_{i=0}^{j - 1} \alpha_{\max}^{(i)} \right)^2 
				= O(\max_{i \in [0   :  j]} (\alpha_{\max}^{(i)})^2)	\, ,
		\end{aligned}
	\end{equation}
where we used $\n{\d_i} = 1$ by assumption in the second equality. On the other hand 
\begin{equation} \label{eq:pscal}
	\Sc{\g}{\tilde{\y}_{j + 1} - \y_0} =  \alpha_{\max}^{(j)} \d_j +  \sum_{i=0}^{j - 1} \alpha_i \Sc{\g_i}{\d_i} = O(\max\limits_{i \in [0   :  j]} \alpha_{\max}^{(i)}) \, .
\end{equation}
Since $\y_{(i)} \in \OM \sm \OM_*$, by the active set related property $\alpha_{\max}^{(i)} = o(1)$ for $\x_k \rightarrow \x_*$. Let now $M_1$ and $M_2$ be the implicit constants in \eqref{eq:Lnorm} and \eqref{eq:pscal}. For $\y_0 = \x_k$ close enough to $\x_*$ we obtain
\begin{equation} \label{eq:pscal2}
	L\n{\tilde{\y}_{j + 1} - \y_0}^2 \leq M_1 \max_{i \in [0   :  j]} (\alpha_{\max}^{(i)})^2 < M_2 \max_{i \in [0   :  j]} \alpha_{\max}^{(i)} \leq \Sc{\g}{\tilde{\y}_{j + 1} - \y_0} \, ,
\end{equation}
where we used \eqref{eq:Lnorm} in the first inequality, $\max_{i \in [0   :  j]} \alpha_{\max}^{(i)} = o(1)$ in the second inequality and \eqref{eq:pscal} in the last inequality. From \eqref{eq:pscal2}, $\tilde{\y}_{j + 1} \in \tx{int}\bgl$ follows easily as desired. \\
We now need to check $\tilde{\y}_{j + 1} \in B_j$. Reasoning as above, on the one hand we have $\frac{\n{\g}}{2L} = \Theta(1)$ for $\x_k \rightarrow \x$ (setting aside the trivial case where $-\nabla f(\x_*) = 0$), and on the other hand $\n{\tilde{\y}_{j + 1} - \y_0} = o(1)$ by \eqref{eq:Lnorm}, so that 
\begin{equation} \label{eq:pscal3}
	\n{\tilde{\y}_{j + 1} - \y_0} < \frac{\n{\g}}{2L}
\end{equation}
for $\y_0$ close enough to $\x_*$ and $\tilde{\y}_{j + 1} \in B_j$ as desired. Then $\tilde{\y}_{j + 1} \in \tx{int}(\OM_j)$ for $j \in [0 \! : \! T^*]$, or equivalently $\beta_j > \alpha_{\max}^{(j)}$ the SSC does always maximal steps in the first $T^* + 1$ iterations. In particular, it generates the point $ \y_{T^* + 1} \in \OM_* \setminus\{ \y_{T^*}\}$. The claim is thus proved. \\
If $\y_{T^* + 1}$ is stationary for $g$, the SSC terminates at step 4 with output $\y_{T^* + 1} \in \OM_*$ and the thesis is proved. Otherwise, we claim that the SSC terminates with output $\y_{T^* + 2} \in \OM_*$ and $\beta_{T^* + 1} < \alpha_{\max}^{(T^* + 1)}$. First, observe that by assumption we must have $\d_{T^* + 1} \in S_*$, and therefore $\y_{T^* + 2} = \y_{T^* + 1} + \alpha_{T^* + 1} \d_{T^* + 1} \in \OM_*$. Second, we have $\alpha_{\max}^{(T^* + 1)} = \Theta(1)$, and at the same time
\begin{equation}
	\beta_{T^* + 1} \leq \tx{diam}(\OM_{T^* + 1}) \leq \tx{diam}(B_{T^* + 1}) \leq 2\Sc{\g_{T^* + 1}}{\d_{T^* + 1}} = o(1) \, .
\end{equation}
	Thus for $\y_0$ close enough to $\x_*$ we must have
	\begin{equation} \label{pscal3}
		\beta_{T^* + 1} < \alpha_{\max}^{(T^* + 1)} \, ,
	\end{equation}
and the claim is proved. Since the SSC terminates either with $\y_{T^* + 1}$ or $\y_{T^* + 2}$, and both of these points are in $\OM_*$, the thesis follows. \qed
\end{proof}
\begin{lemma} \label{l:maxO1}
For $\x \rightarrow \x_*$, $\g \rightarrow -\nabla f(\x_*)$, if  $\x \in \OM_*$, $\d \in S_*$ and $\alpha_{\max}(\x, \hat{\d})$ coincides with the maximal feasible stepsize, then $\alpha_{\max}(\x, \hat{\d}) = \Theta(1)$. 
\end{lemma}
\begin{proof}
	We have
	\begin{equation}
		\alpha_{\max}(\x, \hat{\d}) = \min_{i: \hat{ d }_i < 0} \frac{x_i}{|\hat{ d }_i|} \geq \min_{i: \hat{ d }_i < 0} x_i \geq \min_{i \in \supp(\x_*)} x_i = \Theta(1) \, ,
	\end{equation}
where we used $|\hat{ d }_i| \leq \n{\hat{\d}} \leq 1$ in the first inequality, $\supp(\d) \subseteq \supp(\x_*)$ in the second inequality, and $x_i \rightarrow x_{*, i} > 0$ in the third one. \qed
\end{proof}
For $\x \in \OM$, we define the expression
	$$ \lambda^{(i)} (\x,\g) =  \Sc {\g^{(i)} }{\x^{(i)}}\, \e^{(i)}-\g^{(i)} \, ,\quad i \in [1 \! : \! m]\, ,$$
	and  the
	Lagrangian multiplier vector
	\begin{equation}
		\lambda^{(i)}(\x) = \lambda^{(i)} (\x,-\nabla f(\x)) =\, \nabla f(\x)^{(i)}  - \Sc{\nabla f(\x)^{(i)}}{\x^{(i)}}\, \e^{(i)} \, ,\quad i \in [1 \! : \! m]\, .
	\end{equation}
We notice that strict complementarity holds at a stationary point $\x_* \in \OM$ for $\nabla f(\x_*)$ if and only if it holds for every $i \in [1 \! : \! m]$ at $\x_*^{(i)}\in \OM^{(i)}$ and $\nabla f(\x_*)^{(i)}$.
\begin{lemma} \label{l:AFWs}
Assume that strict complementarity holds at $\x_*$. Then the AFW applied to the simplex has active set related directions in $\x_*$ as in Definition~\ref{def:asrel}.
\end{lemma}
\begin{proof}
	For $\x \rightarrow \x_*$ and $\g \rightarrow -\nabla f(\x_*)$ we have $\lambda (\x, \g) \rightarrow \lambda(\x_*)$, and therefore in particular $\lambda_i(\x, \g) \rightarrow 0$ for $i \in \supp(\x_*)$ while $\lambda_i(\x,\g) \rightarrow \lambda_i(\x_*) > 0$ for $i\in [1 \! : \! n]\sm  \supp (\x_*)$. Therefore, for $\x \in \OM \sm \OM_*$ close enough to $\x_*$ we must have $\max \{\lambda_i(\x, \g) :i \in \supp(\x)\} > \max\{-\lambda_i(\x, \g) :i \in [1 \! : \! n]\}$, so that by \cite[Lemma 3.2(a)]{bomze2020active} we have that the descent direction selected by the AFW satisfies $\d = \x - \e_{\hat{i}}$ for some $\hat{i} \in \argmax \{\lambda_i(\x, \g) :i \in \supp(\x)\} \subset [1 \! : \! n] \sm \supp(\x_*)$. Therefore
	\begin{equation}
		\Sc{\g}{\hat{\d}} = \Sc{\g}{\frac{\x - \e_{\hat{i}}}{\n{\x - \e_{\hat{i}}}}} = \frac{\lambda_{\hat{i}}(\x, \g)}{\n{\x - \e_{\hat{i}}}} = \Theta(1)
	\end{equation}
for $\x \rightarrow \x_*$ and $\g \rightarrow -\nabla f(\x_*)$. \\
As for the case $\x \in \OM_*$, then if $\x, \g$ are close enough to $\x_*$ we must have $\lambda_i(\x, \g) > 0$ for every $i$ in $[1 \! : \! n] \sm \supp(\x_*)$. Therefore by \cite[Lemma 3.2(b)]{bomze2020active} if $y$ is obtained from $\x$ with a FW update we must have $\y_i = 0$ for $i \in [1 \! : \! n] \sm \supp(\x_*)$, which is equivalent to say that the update direction must be in $S_*$. The property $\alpha_{\max}(\x, \hat{\d}) = \Theta(1)$ follows by Lemma \ref{l:maxO1}.  \qed
\end{proof}

\begin{proof}[of Theorem \ref{th:actid}]
Follows by applying the property proved in Lemma \ref{l:AFWs} to each block selected by the method. \qed
\end{proof}

\section{Declarations}

\subsection{Funding and/or Conflicts of interests/Competing interests}
Nothing to declare by all of the authors.

\subsection{Compliance with Ethical Standards}

The entire research work and writing of this article was performed under strict compliance with widely accepted ethical standards by all authors.

\bibliographystyle{spmpsci} 
\bibliography{pdbib}
	
\end{document}